\documentclass[12pt]{article}
\usepackage[utf8]{inputenc}
\usepackage{amsmath}
\usepackage{amsthm}
\usepackage{amsfonts}
\usepackage{amssymb}
\usepackage{algorithm}
\usepackage{algorithmicx}
\usepackage{multirow}
\usepackage{algpseudocode}
\usepackage{graphicx}
\usepackage{cite}
\usepackage{url}
\usepackage{color}
\usepackage{soul}
\usepackage[percent]{overpic}
\newcommand{\citep}[1]{\cite{#1}}
\newcommand{\citet}[1]{\cite{#1}}

\algdef{SE}[DOWHILE]{Do}{doWhile}{\algorithmicdo}[1]{\algorithmicwhile\ #1}%

\newcommand{\abs}[1]{\left| #1 \right|}
\newcommand{\norm}[1]{\left\Vert #1 \right\Vert}
\newcommand{\jmp}[1]{\left[#1\right]}
\newcommand{\avg}[1]{\lbrace #1 \rbrace}
\newcommand{\conj}[1]{\overline{#1}}

\newcommand{\Thetadef}{\Theta_{K,\mathrm{def}}}
\newcommand{\Thetapre}{\Theta_{\mathrm{pre}}}
\newcommand{\bigO}[1]{\mathcal{O}(#1)}

\newtheorem{thm}{Theorem}

\usepackage{array}
\newcolumntype{L}[1]{>{\raggedright\let\newline\\\arraybackslash\hspace{0pt}}m{#1}}

\usepackage[left=2cm,right=2cm,top=2cm,bottom=2cm]{geometry}
\title{{A ray-based IPDG method for high-frequency time-domain acoustic wave propagation in inhomogeneous media}
}
\author{
Eric T. Chung
\thanks{Department of Mathematics, The Chinese University
of Hong Kong. {Email: tschung@math.cuhk.edu.hk.
Eric Chung's research is partially supported by Hong Kong RGC General Research Fund (Project: 14317516)
and CUHK Direct Grant for Research 2016/17.}}\and
Chi Yeung Lam
\thanks{Department of Mathematics, The Chinese University
of Hong Kong. {Email: cylam@math.cuhk.edu.hk.}}\and
Jianliang Qian
\thanks{Department of Mathematics, Michigan State University, East Lansing, MI 48824. Email: qian@math.msu.edu}}
\begin{document}

\maketitle

\begin{abstract}
The numerical approximation of high-frequency wave propagation in inhomogeneous media is a challenging problem. 
In particular, computing high-frequency solutions by direct simulations requires several points per wavelength for stability and usually requires many points per wavelength for a satisfactory accuracy. 
In this paper, we propose a new method for the acoustic wave equation in inhomogeneous media in the time domain to achieve superior accuracy and stability without using a large number of unknowns. 
The method is based on a discontinuous Galerkin discretization together with carefully chosen basis functions. 
To obtain the basis functions, we use the idea from geometrical optics and construct the basis functions by using the leading order term in the asymptotic expansion. 
Also, we use a wavefront tracking method and a dimension reduction procedure to obtain dominant rays in each cell. 
We show numerically that the accuracy of the numerical solutions computed by our method is significantly higher than that computed by the IPDG method using polynomials.
Moreover, the relative errors of our method grow only moderately
as the frequency increases.
\end{abstract}

\section{Introduction}
In this paper, we consider the acoustic wave equation in inhomogeneous media given by
\begin{eqnarray}
	u_{tt}(x,t) - c^2(x)\Delta u(x,t) = 0, \quad (x,t) \in \mathbb{R}^2\times[0,T], \label{acoustic}
\end{eqnarray}
with initial conditions 
$
	u(x,0) = \sum_{l=1}^L A_l(x) e^{i\omega \phi_l(x)}
$
and
$
	u_t(x,0) = \sum_{l=1}^L i\omega B_l(x) e^{i\omega \phi_l(x)},
$
where $c(x)>0$ is the wave speed in the medium, $T>0$ is a given time and $\omega$ is the frequency of the input. 
Due to the use of the theory of geometrical optics in the construction of basis functions, we assume that the initial condition $u(x,0)$
is a superposition of functions in the form $A(x) e^{i\omega \phi(x)}$ for some smooth functions $A(x)$ and $\phi(x)$ independent of $\omega$,
and the initial condition $u_t(x,0)$ is a superposition of functions 
in the form $i\omega B(x) e^{i\omega \phi(x)}$ where the amplitude function $B(x)$ is also a smooth function independent of the frequency $\omega$.
We note that, for more general cases, one can use the technique of micro-local analysis
to express the initial conditions as a superposition of functions in the form $A(x) e^{i\omega \phi(x)}$ \cite{qiayin10a,qiayin10b}. 
We are interested in high-frequency solutions propagating in a medium with smooth wave speed $c(x)$,
in the sense that the variation within each cell of an underlying domain partition is small. 
Thus, we will assume that the frequency $\omega \gg 1$ is a very large number.

\subsection{Our work}
We consider high-frequency solutions of the problem \eqref{acoustic} given by a superposition of wave components:
\begin{eqnarray}
	u(x,t) \approx \text{superposition of } \Big\{ A_k(x,t) e^{i\omega \phi_k(x,t)} \Big\}_{k=1}^N,\label{WKB}
\end{eqnarray}
where $\omega \gg 1$ is the base frequency and $\phi_k$ are the phase functions that satisfies
$$(\phi_{k})_t(x,t) \pm \abs{\nabla\phi_k(x,t)} c(x) =0.$$
In our method, we assume that the expansion (\ref{WKB}) holds in each cell of an underlying partition of the computational domain,
and the number of waves $N$ can vary from cell to cell. 
One can use evenly-distributed ray directions together with plane wave type basis functions. For good approximations, one needs to use many ray directions
but this makes the Galerkin method inefficient and ill-conditioned.

In this paper, 
we consider a modified \emph{wavefront-tracking method}\cite{vinje1993traveltime,lambare1996two,bulant1999interpolation} to capture possible phases in the solution before the actual
simulation.
Since the number of phases can be large in general, we assume that there is a small number of phases that contribute most to the solution. 
Then, we apply a clustering procedure and a dimension reduction procedure to this set of all possible phases and obtain the dominant phases $\phi_k(x)$ in the solution.
Using these dominant phases, we define our basis functions as
$A_k(x) e^{i\omega \phi_k(x)}$,
where $A_k$ are polynomials. We remark that these basis functions are obtained for each cell in the partition of the domain. 
Finally, we use these basis functions together with an interior penalty discontinuous Galerkin (IPDG) method and obtain the approximate solution by solving the resulting linear system. 
We note that the basis functions can also be used for other DG discretizations such as \cite{chung2006optimal,chung2009optimal}. 
We also remark that the wavefront-tracking method can be implemented in a parallel way effortlessly. Moreover, we propose an online/offline scheme to accelerate the simulations for a given medium with many initial conditions,
and this situation arises in many practical applications involving inversions. 
Notice that we will consider smooth media in this paper.
This means that the computational mesh is fine enough to capture the variations in the media. 
For media with more oscillations that cannot be captured by the computational grid, one has to apply some type of multiscale ideas, such as \cite{chung2014generalized,gao2015numerical,gao2015generalized,chung2016mixed}. 
By using our proposed scheme, we are able to compute the solution of the acoustic wave equation in the high-frequency regime in a very efficient way.
Our numerical results show that our scheme is robust on the frequency, i.e. the errors do not grow significantly with the frequency. 
Furthermore, we compare our scheme with the IPDG scheme that uses the standard polynomial basis.
We observe that, with about the same number of unknowns, our scheme performs much better than standard schemes. 
To the best of our knowledge, our method is the first one that combines ray-based idea
and Galerkin method to solve time-dependent wave equations with high frequency solutions.

\subsection{Related works}
In literature, there are some works that focus on 
the closely related problem given by the Helmholtz equation
\begin{eqnarray}
\Delta v(x) + \frac{\omega^2}{c^2(x)}v(x) = 0,\label{helmholtz}
\end{eqnarray}
with frequency $\omega \gg 1$.

On one hand, due to the Shannon's sampling theorem, the minimum requirement for the number of degrees of freedom of representing the solution is of $\bigO{\omega^d}$.
On the other hand, the pollution effect haunts the standard Galerkin methods using low order polynomials as the basis \cite{babuska1997pollution, ihlenburg2006finite}, i.e. the $L^2$-error grows as $\omega$ getting large while keeping $\omega h$ fixed.
A common strategy for solving this problem is to incorporate oscillatory functions into the basis for the Galerkin methods.
This approach significantly reduces the number of degrees of freedom required for accurately representing the solution.

Mainly, there are two types of methods using this strategy.
One type of methods does not assume knowledge of the propagation directions of the wave-field.
Instead, they incorporate analytic or approximate solutions with predefined directions into the basis.
For example, the generalized finite element method \cite{babuvska1995generalized, melenk1995generalized} uses basis function given by the product of plane waves with uniformly-spaced directions and finite element basis functions.
The Trefftz methods use local solutions of the Helmholtz equation as the basis functions\cite{hiptmair2015survey,egger2015space}.
In particular, a common choice of Trefftz basis for a homogeneous medium is the plane waves, for example, as used in the plane wave discontinuous Galerkin method \cite{gittelson2009plane,hiptmair2011plane,hiptmair2016plane} and the discontinuous enrichment method \cite{Farhat20016455}.
For an inhomogeneous medium, some recent papers have proposed methods to construct the basis consisting of approximate local solutions of the form $e^{P(x)}$ with a complex polynomial $P$\cite{imbert2014generalized,imbert2015numerical,imbert2015interpolation}.

Another type of methods relies on an asymptotic approximation, i.e. the WKB ansatz.
This ansatz assumes that a high-frequency solution of the Helmholtz equation \eqref{helmholtz} can be locally approximated by a superposition of wave components,
\begin{eqnarray}
	v(x) \approx \text{superposition of } \Big\{ A_k(x)e^{i\omega\phi_k(x)}  \Big\}_{k=1}^N,
\end{eqnarray}
where $A_k$ and $\phi_k$ are non-oscillatory and independent of $\omega$. The function $A_k$ and $\phi_k$ are called the amplitude and phase, respectively.
A comprehensive survey of the techniques arising from this ansatz can be found in \cite{engquist2003computational}.
These methods incorporate the phases $\phi_k$ in the basis of a Galerkin method and usually called phase-based methods.
One fundamental difficulty of these methods is how the phases $\phi_k$ should be approximated.
For example, a phase-based hybridizable discontinous Galerkin (HDG) method is proposed in \cite{nguyen2015phase}. This method uses basis consisting of the product of polynomials and oscillatory functions of the form $e^{i\omega\phi(x)}$, where the phases $\phi$'s are obtained from the eikonal equation or ray-tracing.
Later, \cite{lam2016phase} analyzes the $h$-convergence of a phase-based IPDG method and shows that, under some conditions, these basis functions have the same approximation power as the polynomials in the product.
Instead of incorporating the global phases, \cite{betcke2012approximation} proposed to use the product of polynomials and plane wave with dominant direction as the basis function for an inhomogeneous medium. 
Numerical results show that such basis performs significantly better than uniformly-spaced plane waves regarding efficiency and accuracy.
Recently, a ray-based finite element method (ray-FEM) is proposed to learn and incorporate dominant directions into the basis in a stable, efficient and systematic way\cite{fanqiazepzha17}. Numerical tests show that this method achieves asymptotic convergence as $\bigO{\omega^{-\frac{1}{2}}}$ when $\omega \to \infty$.

This paper is organized as follows. In Section 2, we give an overview of our proposed method.
In Section 3, we give the detailed algorithms of our method.
In addition, we discuss some computational issues, including parallel computation and conditioning of the resulting linear system.
In Section 4, we present some numerical experiments to study the behavior of the error  as $\omega \to \infty$ and also mesh size $h\to 0$.
Finally, we conclude the paper in Section 5.

\section{The ray-construction based IPDG method}
In this section, we present the key ingredients of
our ray-construction based IPDG method for the high frequency acoustic wave propagation problem (\ref{acoustic}). The algorithmic details are presented in Section \ref{sec:algorithms}.
The main idea of our approach is to
approximate the wave-field via a superposition of plane waves with only dominant ray directions in the solution during the simulation.
We will combine the idea of wavefront tracking and a dimensional reduction procedure to construct the dominant ray directions in the solution.
We then
incorporate these dominant ray directions to form 
plane wave basis functions for an IPDG method to improve the accuracy and stability for
computation of high-frequency solutions.

Consider a cell of an underlying partition of the domain and assume 
the ansatz \eqref{WKB} in this cell. Note that $N$ is the number of phases and can be large in general. It also varies in time. 
Recall that the amplitude $A_k$ and phase $\phi_k$ are independent of the frequency $\omega$.
We take a point $x_0$ in the center of the cell with size $h$.
Using Taylor expansions of $\phi_k$ and $A_k$ in $x$ around the point $x_0$, we have
\begin{equation}
\begin{split}
	u_k(x,t) = & \left(A_k(x_0,t)+\nabla A_k(x_0,t) \cdot (x-x_0)\right) e^{i\omega (\phi_k(x_0,t) + \nabla \phi_k(x_0,t)\cdot(x-x_0))} \\ & \quad + \bigO{h^2 + \omega h^2 + \omega^{-1}},
	\end{split}
\label{WKBappr}
\end{equation}
for $\abs{x-x_0} < h \ll 1$. 
When $\omega \to \infty$ and $\omega h = \bigO{1}$, the asymptotic error decreases as $\bigO{\omega^{-1}}$.
This suggest that we can represent each component of the solution locally at $x_0$ by
the product of a linear function and a plane wave {$e^{i\omega\nabla \phi_k(x_0,t)\cdot (x-x_0)}$}.
Moreover, we essentially need $\nabla \phi_k(x_0,t)$ to determine the polynomial-modulated plane wave in the local approximation. In this paper, 
our aim is to find the {\bf ray directions}
$\nabla \phi_k(x_0,t)$
and determine the {\bf dominant ray directions} within the set of ray directions. 

Now, suppose there is a small error $\varepsilon$ in the vector $\nabla \phi_k(x_0,t)$. Then the asymptotic error in \eqref{WKBappr} becomes $\bigO{h^2+ \omega h \varepsilon+\omega h^2 + \omega^{-1}}$. Also, when $\omega \to \infty$ and $\omega h = \bigO{1}$, the asymptotic error decreases as $\bigO{\varepsilon+\omega^{-1}}$.
Our approach is to replace all the $\nabla \phi_k(x_0,t)$ by some approximations, which are easy to compute. 
In exchange, the number of degrees of freedom could be significantly reduced and the resulting system would be better conditioned.
Despite the extra $\bigO{\varepsilon}$ term in the ansatz, our numerical results in Section 5 confirm that our method still benefits a lot from such approximation.

To simplify our discussion, from now on we restrict our discussion to the domain {$\Omega = (0,1)^2$},
and assume {that} the solution $u(x,t)$ satisfies the periodic boundary condition. 
We divide the domain $\Omega$ into $N\times N$ square cells $\mathcal{T}_h$, given by $K_{ij} := [\frac{i-1}{N}, \frac{i}{N}]\times[\frac{j-1}{N}, \frac{j}{N}],$ for $1\leq i,j \leq N$,
where $h = 1/N >0$ is the mesh size.
We let $\mathcal{F}_h$ be the set of faces of the partition. 
For a given cell $K \in\mathcal{T}_h$,
we define $x_K$ to be the centroid of $K$. We call these points the \emph{observation points}.	
For each $x_K$, we define the set $\Theta_K$ of all possible ray directions by
$$
\Theta_{K} := \left\lbrace \nabla \phi_k(x_K, t) :  1\leq k\leq N_K, \,0 \leq t \leq T \right\rbrace
$$
where $N_K$ is the number of phases, which depends on the cell $K$ and also the time $t$. 
We remark that we assume the solution contains a finite, and possibly large,
number of phases at any time instant.

Next, we use the symmetric interior-penalty type discontinuous Galerkin (IPDG) method \cite{grote2006discontinuous} to compute the solution of \eqref{acoustic}.
For $j = 1,2,3,4$, let $v_{j,K}$ be the four vertices of $K$,
and {$\varphi_{j,K}$} be the standard Lagrange-type 
bilinear basis on $K$ such that {$\varphi_{j,K}(v_{i,K}) = \delta_{ij}$}, where $\delta_{ij}$ is the Kronecker delta. 
We define an enriched local approximation space by
$$
V(\Theta_K) = \mbox{span} \left\lbrace \varphi_{j,K} e^{i\omega p\cdot (x-x_K)} : p\in \Theta_K, \; j = 1,2,3,4 \right\rbrace
$$
where functions in $V(\Theta_K)$ are defined on the cell $K$ only. 
Then we can define the global approximation space 
\begin{equation*}
V_c = \cup_{K\in\mathcal{T}_h} V(\Theta_K).
\end{equation*}
Following the derivation of the standard IPDG method,
we can write down the following semi-discrete scheme: find $u_h \in V_c$ such that 
$$
\left( \frac{\partial^2 u_h}{\partial t^2} ,v_h \right) + a^{\gamma}_h(u_h, v_h) = {0},
$$
for any $v_h\in V_c$, where $(u,v) =\int_{\Omega} c^{-2}\, u\conj{v}\, dx$ and 
the bilinear form $a^{\gamma}_h$ is given by
\begin{equation}
\begin{split}
a^{\gamma}_h(u,v) := & \int_{\Omega} \nabla u \cdot \conj{\nabla v}\, dx
 - \sum_{F\in\mathcal{F}_h} \int_F \avg{\nabla u \cdot n}\, \conj{\jmp{v}}\, ds 
 - \sum_{F\in\mathcal{F}_h} \int_F {\jmp{u}}\,\conj{\avg{\nabla v \cdot n}}\, ds\\ \qquad & 
 + \frac{\gamma}{h} \int_F \jmp{u}\conj{\jmp{v}}\, ds,
 \label{a_h}
\end{split}
\end{equation}
for any $u,v\in V$, where
$\gamma >0$ is the penalty parameter. 
Here $[\,\cdot\,]$ and $\lbrace\,\cdot\,\rbrace$ are the usual jump and average operators in discontinuous Galerkin methods, which are defined in the following way.
Let $K^{\pm}$ be two tiles sharing an edge $F$, $n^{\pm}$ be the outward normal of $K^{\pm}$ on $F$ and $u^{\pm}$ be the two smooth scalar functions on $K^{\pm}$.
The average operator $\avg{\,\cdot\,}$ and jump operator $\jmp{\,\cdot\,}$ are given by
$$
\avg{u} := \frac{1}{2}(u^+ + u^-) \mbox{\quad and\quad } \jmp{u} := u^+n^+ + u^-n^-,
$$
respectively.
Similarly, for smooth vector {fields} ${\boldsymbol\sigma}^{\pm}$ defined on $K^{\pm}$, respectively,
the average operator $\avg{\,\cdot\,}$ and jump operator $\jmp{\,\cdot\,}$ are given by
$$
\avg{\boldsymbol\sigma} := \frac{1}{2}(\boldsymbol\sigma^+ + \boldsymbol\sigma^-)
\mbox{\quad and\quad }
\jmp{\boldsymbol\sigma} := \boldsymbol\sigma^+\cdot n^+ + \boldsymbol\sigma^-\cdot n^-.
$$

We note that the set $\Theta_K$ contains a continuum of ray directions, and thus cannot be used directly for computations. 
With this in mind, we will 
construct a finite set $\widetilde{\Theta}_{K}$, which is a subset of $\Theta_K$
and contains all dominant ray directions within the cell $K$. 
For this purpose, we will consider a partition of $[0,T]$ with time step $\Delta t$,
and we denote $t_m = m \Delta t$, $m=0,1,\cdots$. Then we consider the set
\begin{equation*}
\widehat{\Theta}_{K} := \left\lbrace p : p = \nabla \phi_k(x_K, t_m), \; \text{for all } k, t_m \text{ satisfying } 1\leq k\leq N_K, 0\leq t_m \leq T \right\rbrace.
\end{equation*}
We note that the set $\widehat{\Theta}_{K}$ is finite and contains all ray directions resolved by the partition in time.
We emphasize that this set has a large dimension, and cannot be used directly in simulations. 
One key ingredient of our method is to perform a dimension reduction for the set $\widehat{\Theta}_{K}$
and obtain a subset $\widetilde{\Theta}_{K}$ with a much smaller dimension. 
In particular, the set $\widetilde{\Theta}_{K}$ has the form
\begin{equation*}
\widetilde{\Theta}_{K} := \left\lbrace p : p = \nabla \phi_k(x_K, t_m), \; \text{for some } k, t_m \text{ satisfying } 1\leq k\leq N_K, 0\leq t_m \leq T \right\rbrace.
\end{equation*}
We will give a detail discussion on how to obtain this set in the next section.
Next, we define an enriched local approximation space by
$$
V(\widetilde{\Theta}_K) = \mbox{span} \left\lbrace \varphi_{j,K} e^{i\omega p\cdot (x-x_K)} : p\in \widetilde{\Theta}_K, \; j = 1,2,3,4 \right\rbrace
$$
where functions in $V(\Theta_K)$ are defined on the cell $K$ only. 
Then we can define the global approximation space 
\begin{equation*}
V = \cup_{K\in\mathcal{T}_h} V(\widetilde{\Theta}_K).
\end{equation*}
Using the space $V$, we can write down the fully-discrete scheme: 
Given $u^0_h \in V$ and $u^1_h \in V$, for $n\Delta t < T$, we find $u_h^{n+1} \in V$ such that
$$
\left( \frac{u^{n+1}_h - 2 u^{n}_h + u^{n-1}_h}{\Delta t^2} ,v_h \right) + a^{\gamma}_h(u^n_h, v_h) = {0},
$$
for any $v_h\in V$. 
We remark that, by using a standard stability analysis, the time step $\Delta t$
should be chosen such that $\Delta t \| \mathbf{A}\|_2 < 1$,
where $\mathbf{A}$ is the stiffness matrix resulting from the bilinear form $a_{\gamma}$
and $\|\cdot\|_2$ denotes the $2$-norm. 

The above gives a general outline of our scheme. In the next section,
we present the detailed implementations.

\section{Algorithms}
\label{sec:algorithms}
In this section, we present the full algorithm for the ray based IPDG method.
We divide it into two conceptual stages:
\begin{enumerate}
\item (ray-construction stage) determining the {dominant} rays in the solution, i.e. the set $\widetilde{\Theta}_{K}$ at every observation point $x_K$.
\item (time marching stage) solving the IPDG system using the approximate space $V(\widetilde{\Theta}_K)$.
\end{enumerate}

For the phase-construction stage, we use a modified version of the wavefront tracking method proposed in \cite{vinje1993traveltime} to determine all ray directions of the solution throughout $0\leq t\leq T$.
In particular, we will obtain approximations to the ray directions at the observation point $x_K$ for the cell $K$ in the domain partition. 
We call the set of all these ray directions $\widetilde{\Theta}_K^*$.
Afterward, we apply a clustering procedure on $\widetilde{\Theta}_K^*$
and obtain our desired set $\widetilde{\Theta}_K$ of dominant ray directions. 

In the following Sections 3.1--3.4, we discuss the procedures for the phase-construction stage.
In Section 3.5, we integrate the ideas from Sections
3.1--3.4 into an algorithm for the ray-construction stage.
In Section 3.6, we write down an algorithm for the time marching stage.
In Section 3.7, we propose an online/offline scheme based on the ray-construction based IPDG method.

\subsection{Wavefront propagation}
In this procedure, we compute the wavefront propagation and use it to determine the ray directions. 
First of all we define wavefronts, which are essentially the level sets of the phase functions, together with the gradients of the phase functions on the level sets. 
More precisely, we let
$\phi$ be one of the phases $\phi_k$ in the ansatz \eqref{WKB}.
A \emph{wavefront} $W(t_0,\alpha)$, at a fixed time instant $t_0$,
corresponding to the phase function $\phi(x,t_0)$, is essentially the level set curve $\phi(x,t_0)=\phi_0$ for a given constant $\phi_0$, which may depend on $t_0$.
Mathematically, the wavefront $W(t_0,\alpha)$
is a smooth curve $(y(t_0,\alpha), q(t_0,\alpha)) \in \mathbb{R}^2 \times \mathbb{R}^2$, parametrized by $\alpha \in J$,  
such that for any $\alpha$,
\begin{equation}
\label{eq:phase}
	\phi(y(t_0,\alpha), t_0) \equiv \phi_0 \mbox{ and } q(t_0,\alpha) = \nabla \phi(y(t_0,\alpha), t_0)
\end{equation}
where $J$ denotes an interval of real numbers. 
That is, the component $y(t_0,\alpha)\in\mathbb{R}^2$ lies on the level set curve $\phi(x,t_0)=\phi_0$,
and that the level set curve $\phi(x,t_0)=\phi_0$ is represented by the function $y(t_0,\alpha)$ and is parametrized by $\alpha$. 
In addition, the component $q(t_0,\alpha)\in \mathbb{R}^2$ is the gradient vector of the function $\phi(x,t_0)$
at the point $y(t_0,\alpha)$. 

A related concept is a ray.
By inserting $u = A(x,t) e^{i\omega\phi(x,t)} + \bigO{\omega^{-1}}$ into \eqref{acoustic} and considering the leading order term, 
we obtain the following eikonal equation for the phase $\phi(x,t)$:
\begin{eqnarray}
	\phi_t(x,t) + c(x) \abs{\nabla \phi(x,t)} = 0.
\end{eqnarray}
A \emph{ray} $r(t)$ is a bicharacteristic pair $(x(t), p(t)) \in \mathbb{R}^2 \times \mathbb{R}^2$ related to the Hamiltonian $H(x,p) = c(x)\abs{p}$ which is characterized by the following differential equations,
\begin{eqnarray}
	\label{eq:rayODE}
	x_t = c(x) \frac{p}{\abs{p}}, \quad p_t = -\abs{p} \nabla c(x).
\end{eqnarray}
The functions $x(t)$ and $p(t)$ are called the position and ray direction of the ray $r(t)$ respectively.
Rays and wavefronts are related in the following way. 
If a ray $(x(t), p(t))$ lies on a wavefront $W$ initially, that is,
$
(x(0),p(0)) = W(0,\alpha_0) \mbox{ for some } \alpha_0,
$
then the ray lies on the same wavefront for any $t>0$, that is,
$(x(t),p(t)) = W(t, \alpha_0).$

We will approximate a wavefront $W$ at the $n$-th time step $t_n$ by a finite set of points in $\mathbb{R}^2 \times \mathbb{R}^2$, which is called a \emph{discrete wavefront}
and is denoted by $W^n$.
We write $W^n = \lbrace r^n_j \rbrace$. We will parametrize $W^n$ by a discrete set of parameters $\{ \alpha_j\}$.
For example, using the notations in (\ref{eq:phase}), we will take a discrete set $\{ \alpha_j\} \subset J$ and define
\begin{equation*}
r^n_j := ( y(t_n,\alpha_j), q(t_n,\alpha_j)).
\end{equation*}
The wavefronts at the initial time $t=0$ are given by the initial condition of the problem (\ref{acoustic}). 
In particular, we will assume that the phase $\phi(x,0)$ and the ray direction $\nabla \phi(x,0)$ are known at the initial time.
We will use a set of level curves of $\phi(x,0)$ to construct the wavefronts at the initial time. 
The number of these level curves is user-defined. We will assume that the number of these level curves is large enough
so that the ray directions of the initial condition are well resolved by the partition of the computational domain. 
For each of these level curves, we will use (\ref{eq:phase}) to
define a wavefront $W^0$. 
For more precise constructions, see Section \ref{sec:num} where numerical examples are presented.

Suppose that all wavefronts are computed at the $n$-th time step. 
Let $W^n = \lbrace r^n_j \rbrace = \lbrace (y(t_n,\alpha_j),q(t_n,\alpha_j)) \rbrace$ be a given wavefront at the time $t_n$. 
To advance the wavefront $W$ to the $(n+1)$-th time step $t_{n+1}$, we treat each $r^n_j$ as a ray at time $t_n$,
and then take this as the initial condition of (\ref{eq:rayODE})
and compute the solution, via a fourth order Runge Kutta scheme, for one time step (that is, from the time $t_n$ to the time $t_{n+1}$). 
We call this solution as $r^{n*}_j$, and this will be taken as part of the wavefront at the time $t_{n+1}$.
In other words, $r^{n*}_j \approx W\left(t_{n+1}, \alpha_j\right)$.
The pseudocode is given in Algorithm 1.
Note that, the wavefront at the $(n+1)$-st step will be obtained after the reconstruction procedure presented next. 

\begin{algorithm}[h]
\caption{Wavefront propagation}
\begin{algorithmic}[1]
\Procedure{$\lbrace r^{n*}_j\rbrace =$ WavefrontPropagate}{$\lbrace r^n_j\rbrace$, $c$, $\Delta t$}
\ForAll{ $j$}
\State
$(x,p) \gets r^n_j,\quad r^n_{j,1} \gets ( c^2(x)\,p, -\nabla c(x) \abs{p})$
\State
$(x,p) \gets r^n_j + \frac{\Delta t}{2} r^n_{j,1},\quad r^n_{j,2} \gets ( c^2(x)\,p, -\nabla c(x) \abs{p})$
\State
$(x,p) \gets r^n_j + \frac{\Delta t}{2} r^n_{j,2},\quad r^n_{j,3} \gets ( c^2(x)\,p, -\nabla c(x) \abs{p})$
\State
$(x,p) \gets r_j + \Delta t\ r^n_{j,3},\quad r^n_{j,4} \gets ( c^2(x)\,p, -\nabla c(x) \abs{p})$
\State $r^{n*}_j \gets r^n_j + \frac{\Delta t}{6}(r^n_{j,1} + 2 r^n_{j,2} + 2 r^n_{j,3} + r^n_{j,4})$
\EndFor
\EndProcedure
\end{algorithmic}
\end{algorithm}
\subsection{Wavefront reconstruction}
In this procedure, we insert new rays into a discrete wavefront to maintain the quality of the wavefront.
We choose a tolerance function $\mbox{tol}\left(x,p\right) = \alpha_1 \abs{x} + \alpha_2 \abs{p}$ for some constants $\alpha_1$ and $\alpha_2$.
Whenever $\mbox{tol}(r^{n*}_j-r^{n*}_{j+1})\geq 1$ for a pair of neighboring rays $r_j^{n*}$ and $r_{j+1}^{n*}$,
we insert $\lfloor \mbox{tol}(r^{n*}_j-r^{n*}_{j+1}) \rfloor$ new rays between them.
We take the new rays as equidistant linear interpolation of the corresponding neighboring rays.

By connecting each pair of neighboring rays $r_j^{n*}$ and $r_{j+1}^{n*}$ by a line segment, we obtain a linear interpolant in $\mathbb{R}^4$.
We order all $r^{n*}_j$ and the inserted rays along this spline and obtain $W^{n+1} = \lbrace r^{n+1}_j \rbrace$.
We use $j^*$ to denote the new index at the $(n+1)$-th time step corresponding to $r^{n*}_j$.
In other words, $r^{n*}_j$ and $r^{n+1}_{j^*}$ are equal.
It is clear that $j^* \leq \ell \leq (j+1)^*$ for any $r^{n+1}_{\ell}$ inserted between $r^{n+1}_{j^*}$ and $r^{n+1}_{(j+1)^*}$.
The pseudocode of this procedure is shown in Algorithm 2.
Note that, in this algorithm, we also compute
$$
I_j := \max \left\lbrace i : i^*\leq j \right\rbrace,
$$
which will be used in the next procedure.

\begin{algorithm}[H]
\caption{Wavefront reconstruction}
\begin{algorithmic}[1]
\Procedure{$\lbrace r^{n+1}_j \rbrace, \lbrace I_j \rbrace$ = WavefrontRecon}{$\lbrace r^{n*}_j \rbrace$, $\mbox{tol}$}
	\State $j \gets 1$
	\ForAll {$k \gets 1$ \textbf{to} $\left( \mbox{size of } \lbrace r^{n*}_j \rbrace\right)-1$}
		\State $n \gets \lfloor \mbox{tol}(r^{n*}_k - r^{n*}_{k+1}) \rfloor$
		\For {$\ell \gets 0$ \textbf{to} n}
			\State $r^{n+1}_{j+\ell} \gets (1-\frac{\ell}{n+1}) r^{n*}_k + \frac{\ell}{n+1} r^{n*}_{k+1}$
			\Comment{Interpolation}
			\State $I_{j+\ell} \gets k$
		\EndFor
		\State $j \gets j+n$
	\EndFor
	\State $r^{n+1}_{j} \gets r^{n*}_{k+1}, I_{j} \gets k+1$
\EndProcedure
\end{algorithmic}
\end{algorithm}

\subsection{Ray determination}
In this procedure, we determine the ray directions of the solution at an observation point when an approximate wavefront passes through the point.
Let $x_K$ be an observation point. To check whether a wavefront $W$ is passing through $x_K$ between the $n$-th and $(n+1)$-th time step,
we consider the corresponding discrete wavefronts at the time steps, namely, $W^n$ and $W^{n+1}$.
We form triangles in $\mathbb{R}^2$ using the position of the rays in $W^n$ and $W^{n+1}$ and check whether these triangles contains $x_K$.
More precisely, we write $W^n = \lbrace (x^n_j,p^n_j) \rbrace$ and consider the triangles $x^n_j x^{n+1}_{j^*} x^{n+1}_{j^*+\ell +1}$ for $\ell \leq (j+1)^* -j^*-1$, and $x^n_j x^n_{j+1} x^{n+1}_{(j+1)^*}$.
We call each of these triangles a \emph{ray~cell}.

Whenever $x_K$ lies in one of the ray cells, we approximate the ray direction at $x_K$ by 
linear interpolation.
Suppose the rays corresponding to the three vertices of the ray cell are given by $r_1, r_2, r_3$. Then we approximate the ray direction by
$$\mathcal{I}(r_1, r_2, r_3; x_K) = \lambda_1 p_1 + \lambda_2 p_2 + \lambda_3 p_3,$$
where $\lambda_{\ell} \in \mathbb{R}$ is the barycentric coordinates of $x_K$ in the triangle $x_1\, x_2\, x_3$, $x_{\ell}$ and $p_{\ell}$ are the positions and ray directions of the rays $r_{\ell}$ respectively for $\ell = 1,2,3$.
The barycentric coordinates can be obtained by solving
$$
x_K = \lambda_1 x_1 + \lambda_2 x_2 + \lambda_3 x_3 \mbox{ and } \lambda_1+\lambda_2+\lambda_3 = 1.
$$
We name the set of these approximations to the ray directions at $x_K$ by $\widetilde{\Theta}_K^*$.
%
%
%

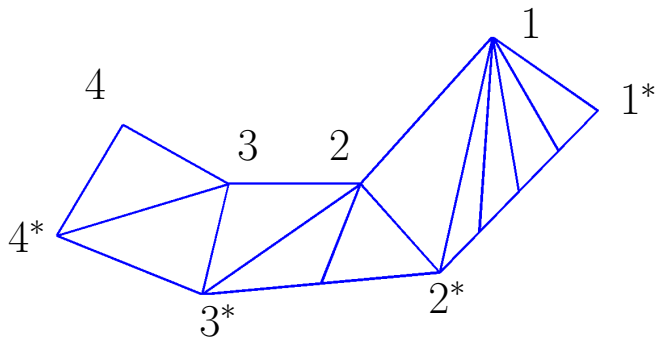
\begin{figure}[h]
\centering
\begin{overpic}[width=4in]{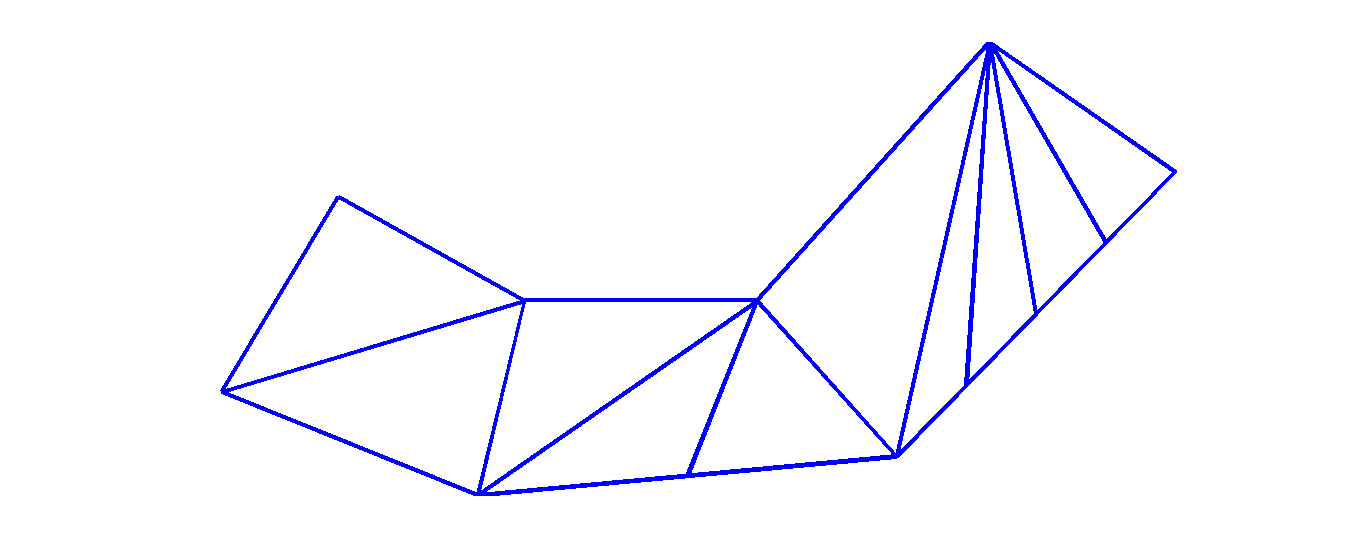}
\put (20,30) {\Large $4$}
\put (40,22) {\Large $3$}
\put (52,22) {\Large $2$}
\put (77,38) {\Large $1$}
\put (10,10) {\Large $4^*$}
\put (35,-1) {\Large $3^*$}
\put (65,2) {\Large $2^*$}
\put (90,28) {\Large $1^*$}
\end{overpic}
\caption{An example of ray cells.
In this figure, we show the position of rays on a wavefront in the physical space at two consecutive time steps, where
the rays $1$--$4$ propagate to the rays $1^*$--$4^*$ at the next time step.
We insert some rays between $1^*$ and $2^*$, and $2^*$ and $3^*$.}
\label{raycell}
\end{figure}
\begin{algorithm}[h]
\caption{Phase determination}
\begin{algorithmic}[1]
\Procedure{$\lbrace \widetilde{\Theta}^*_{K} \rbrace =$ PhaseDet}{$\lbrace r^n_j\rbrace, \lbrace r^{n+1}_j\rbrace, \lbrace I_j \rbrace, \lbrace x_K\rbrace$,
			$\lbrace \widetilde{\Theta}^*_K \rbrace$}
\State $j \gets 1$
\For{ $i \gets 1 $ \textbf{to}  $(\mbox{size of } \lbrace r^n_j \rbrace)-1$}
	\State $j = 1$
	\While{ $I_j + 1 == i$}
		\State $R^n_{j} \gets (r^n_i, r^{n+1}_{j}, r^{n+1}_{j+1})$ \Comment{Form ray cells}
		\State $j \gets j +1$
	\EndWhile
	\State $R^n_{j} \gets (r^n_i, r^n_{i+1}, r^{n+1}_{j})$ 
	\ForAll{ $x_K, R^n_{\ell}$ }
		\State $\left(r_1, r_2, r_3\right) \gets R^n_{\ell}$
		\State $x_{\ell} \gets $ position of $r_{\ell}$, for $\ell = 1,2,3$
		\If{$x_K \in x_1\, x_2\, x_3$}
			\State $\widetilde{\Theta}^*_{K} \gets \widetilde{\Theta}^*_{K} \cup \lbrace \mathcal{I}(r_1, r_2, r_3; x_K) \rbrace$ \Comment{Append phase}
		\EndIf
	\EndFor
\EndFor
\EndProcedure
\end{algorithmic}
\end{algorithm}
\newpage
Note that the time complexity of Algorithm 3 is $\bigO{mn}$ for $n$ observation points and $m$ ray cells.
However, it is possible to reduce the time complexity by using some heuristics.
Supposing that the ray cells are adequately small in size and regular in shape,
we may assume {that} an observation point $x_K$ lies in a ray cell only if all the vertices of a ray cell lies in $K$.
Then if one of the vertices of a ray cell is $x$, the only possible observation point that lies in the ray cell is given by
$\left( \frac{\lfloor Nx \rfloor+\frac{1}{2}}{N}, \frac{\lfloor Ny \rfloor+\frac{1}{2}}{N} \right)$, reducing the {time} complexity to $\bigO{m}$.
Similarly, one can reduce the time complexity of Algorithm 3 for a non-rectangular computational domain to $\bigO{m}$
whenever there is a region classifier, which maps a points to a region, of time complexity $\bigO{1}$.

\subsection{Ray separation}
\label{sec:separation}
In this procedure, we find dominant ray directions in each $\widetilde{\Theta}^*_{K}$
after all the approximate wavefronts have propagated till time $T$.
The aim of this procedure is to obtain a set $\widetilde{\Theta}_{K}$ such that the deviation
$d(\widetilde{\Theta}_{K}, \widetilde{\Theta}^*_{K})$ is not too large while
the minimum distance between any two distinct elements in $\widetilde{\Theta}^*_{K}$, or namely the separability, is not too small.
We will see that using the new set of ray directions $\widetilde{\Theta}_{K}$ will improve the numerical stability
in the IPDG formulation, while the approximation power of the new basis do not deteriorate significantly.

We use a parameter $\varepsilon > 0$ to control the deviation and the separability.
The idea of this procedure is to cover the set $\widetilde{\Theta}^*_{K}$ by some balls with radius $\varepsilon$,
and make sure the centers of the balls are far enough from each other. We take the {dominant ray directions} $\widetilde{\Theta}_K$ by the centers of these balls.
We will also need a predefined set $\Thetadef$ to be included in 
$\widetilde{\Theta}_K$ obtained from this procedure.
One usage of this predefined set is to include the ray directions corresponding to the initial solutions.

We illustrate this procedure in Figure~\ref{fig_phase_sep}.
Let $B_{\varepsilon}(x_0)$ be the Euclidean ball with radius $\varepsilon$ centered at $x_0$.
To begin with, we take $\widetilde{\Theta}_{K}$ to be $\Thetadef$ and eliminate any ray direction in $\widetilde{\Theta}^*_{K}$
that also lies in $\Thetadef + B_{\varepsilon}(0)$.
Then we repeatedly do the following until $\widetilde{\Theta}^*_{K}$ is empty.
Firstly, we take a ray direction $p$ from $\widetilde{\Theta}_{K}$.
Secondly, we compute the centroid $\bar{q}$ of the set $\widetilde{\Theta}^*_{K} \cap B_{\varepsilon}(p)$, 
and include $\bar{q}$ in $\widetilde{\Theta}_{K}$.
Finally, we eliminate any ray direction in $\widetilde{\Theta}^*_{K}$ that also lies in $B_{\varepsilon}(\bar{q})$ and repeat.
The pseudocode is shown in Algorithm 4.

\begin{figure}[h!]
\centering
\includegraphics[height=5cm]{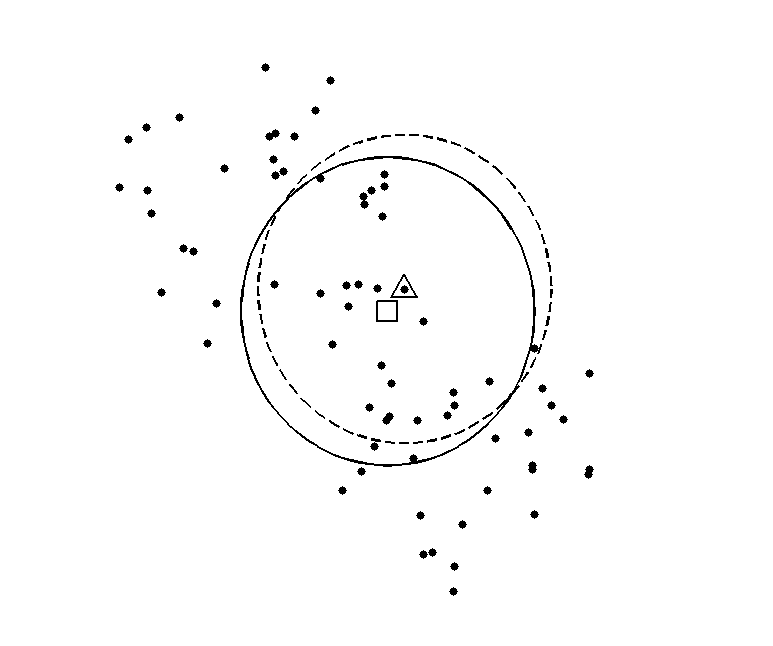}
\includegraphics[height=5cm]{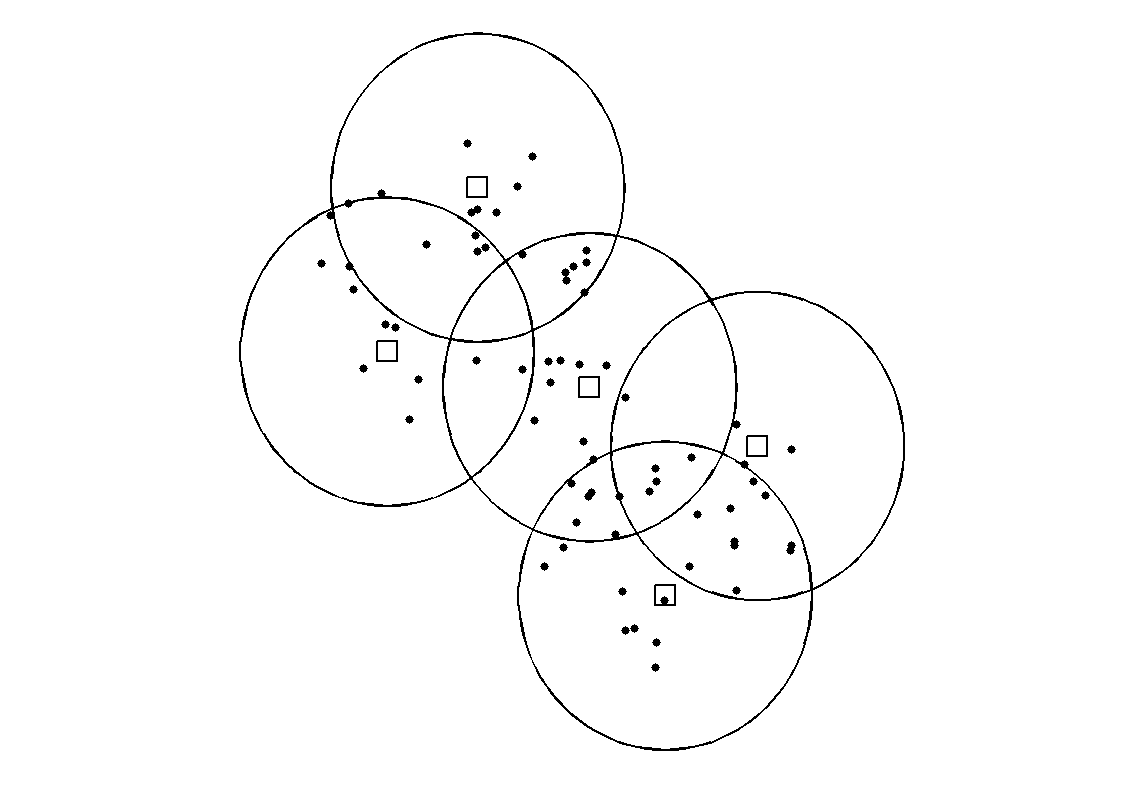}
\caption{
In the left figure, each dot represent a phase and the $\bigtriangleup$ is one of the phases.
The dotted circle is centered at the triangle with radius $\varepsilon$.
The $\square$ is the average of all dots in the dotted circle.
The solid circle is centered at $\square$ with radius $\varepsilon$.
Then we select a dot which does not lie in any solid circle and we repeat the above until such a point does not exist.
Finally we get the right figure.
Note that in this example any two $\square$ are at least $\varepsilon$ apart (in theory only $\varepsilon/2$ apart is achievable) and the distance between a point and the closest $\square$ is less than $\varepsilon$.
}
\label{fig_phase_sep}
\end{figure}

\begin{algorithm}[h]
\caption{Phases separation}
\begin{algorithmic}[1]
\Procedure{$\widetilde{\Theta}_K$ = PhaseSep}{$\widetilde{\Theta}^*_{K}$, $\Thetadef$, $\varepsilon$}
\State $\widetilde{\Theta}_{K} \gets \Thetadef$
\State $\widetilde{\Theta}^*_{K} \gets \widetilde{\Theta}^*_{K} \setminus (\Thetadef + B_{\varepsilon}(0))$
\While{$\widetilde{\Theta}^*_{K} \neq \emptyset$}
	\State $p \gets$ an element of $\widetilde{\Theta}^*_{K}$
	\State $\bar{q} \gets $ centroid of the set $\widetilde{\Theta}^*_{K} \cap B_{\varepsilon}(p)$ 
	\State $\widetilde{\Theta}_{K} \gets \widetilde{\Theta}_{K} \cup \lbrace \bar{q} \rbrace$
	\State $\widetilde{\Theta}^*_{K} \gets \widetilde{\Theta}^*_{K} \setminus B_{\varepsilon}(\bar{q})$
\EndWhile
\EndProcedure
\end{algorithmic}
\end{algorithm}

Next, we discuss some properties of the set $\widetilde{\Theta}_{K}$ obtained from Algorithm 4.
We say that a finite set $\Theta \subset \mathbb{R}^2$ is $\varepsilon$-separable if the distance between any two distinct elements in $\Theta$ is at least $\varepsilon$. The following theorem gives an estimate for the minimum separation of the set $\widetilde{\Theta}_{K}$.

\begin{thm}
If $\Thetadef$ is an $(\varepsilon/2)$-separable set, then the set $\widetilde{\Theta}_{K}$ obtained in Algorithm 4 is an
$(\varepsilon/{2})$-separable set.
\end{thm}
\begin{proof}
Consider $p = (q,0)^T$, $q>\varepsilon$, $\abs{p_j} > \varepsilon$ and $\abs{p_j-p} < \varepsilon$ for $j=1,\ldots, m$.
Then $p_j$ lies in the half-plane $H := \lbrace (x,y) : x>\varepsilon/2 \rbrace$. Since $H$ is convex,
the convex hull of $\lbrace p_j \rbrace_{j=1}^m$, which contains $\frac{1}{m}\sum p_j$, also lies in $H$.
The result follows by translation and rotation of the points in the above argument.
\end{proof}
Note that if we replace line 6 in Algorithm 4 by `$\bar{q} \gets p$', this modified algorithm produces an $\varepsilon$-separable set given that $\Thetadef$ is $\varepsilon$-separable.
Empirically, the set $\widetilde{\Theta}_{K}$ will {contain} more ray directions in this modified version while the minimum separation of the set of ray directions of {the} original version is close to $\varepsilon$.
\subsection{The ray-construction stage}
In this section we combine Algorithms 1--4 to determine
the set of dominant ray directions in each cell $K$ throughout the simulation, namely, $\widetilde{\Theta}_K$.
We give the pseudocode in Algorithm 5.
Here we describe the parameters required for this procedure.
The parameter $c$ and $T$ are the velocity function and the time of simulation defined in the problem \eqref{acoustic}, respectively;
the parameter $\Delta t$ is the time step in the wavefront propagation;
the parameter $\lbrace x_K \rbrace$ is the set of observation points in each cell;
the parameter $\mbox{tol}(x,p)$ is the tolerance function which controls how many points should be inserted between two neighboring rays in the wavefront reconstruction procedure;
the parameter $\lbrace \Thetadef \rbrace$ is the collection of the default set of ray directions used in the phase separation procedure;
the parameter $\lbrace r^0_j \rbrace$ is the collection of discrete wavefront corresponding to the initial ray directions;
and the parameter $\varepsilon$ controls the deviation and separability in the phase separation procedure.
The return sets $\widetilde{\Theta}_K$ are our desired dominant ray directions on $K$.

\begin{algorithm}[h!]
\caption{The phase-construction stage}
\begin{algorithmic}[1]
\Procedure{ $\lbrace \widetilde{\Theta}_{K}\rbrace=$ PhaseConstruction}{$c, T, \Delta t, \lbrace x_K\rbrace, \lbrace \Thetadef \rbrace, \lbrace\lbrace r^0_j \rbrace\rbrace, \mbox{tol}, \varepsilon$}
\ForAll{$K$}
	\State $\widetilde{\Theta}^*_{K} \gets \emptyset$
\EndFor
\ForAll{$\lbrace r^0_{j} \rbrace$}
	\For{ $n \gets 0, 2, \ldots, T/\Delta t-1$ }
		\State $\lbrace r^{(n+1)*}_j \rbrace \gets \textsc{WavefrontPropagate}(c, \Delta t, \lbrace r^{n}_j \rbrace)$
		\State $\lbrace r^{n+1}_j \rbrace, \lbrace I_j \rbrace \gets \textsc{WavefrontRecon}(\lbrace r^{(n+1)*}_j \rbrace,\mbox{tol})$
		\State $\lbrace\widetilde{\Theta}^*_{K}\rbrace \gets \textsc{PhaseDet}(\lbrace x_K \rbrace, \lbrace\widetilde{\Theta}^*_{K}\rbrace, \lbrace{I_j}\rbrace, \lbrace r^n_k \rbrace, \lbrace r^{n+1}_k \rbrace)$
	\EndFor
\EndFor
\ForAll{$K$}
	\State $\widetilde{\Theta}_{K} \gets \textsc{PhaseSep}(\Thetadef, \widetilde{\Theta}^*_{K}, \varepsilon)$
\EndFor
\EndProcedure
\end{algorithmic}
\end{algorithm}

\subsection{The time-marching stage}
In this stage, we obtain the solution of \eqref{acoustic} by solving an IPDG system as described in Section 2. The pseudocode of this stage is given in Algorithm 6. The basis is constructed based on the ray directions obtained in the ray-construction stage. The penalty parameter in the IPDG scheme is denoted by $\gamma$.

\begin{algorithm}[h]
\caption{The time-marching stage}
\begin{algorithmic}[1]
\Procedure{$u_h=$ IPDG-Solve}{$c, \omega, T, \Delta t, h, \lbrace\widetilde{\Theta}_K\rbrace, \gamma,  u^0, u^1$}
	\State $k = 0$
	\For{$i,j \gets 0,1,\dots 1/h-1$}
		\State $K \gets [ih, (i+1)h]\times[jh, (j+1)h]$
		\State $\lbrace \varphi_{j,K} \rbrace \gets \lbrace\mbox{bilinear polynomial basis on } K \rbrace$
	\ForAll{$p \in \widetilde{\Theta}_K,\, \varphi_{j,K}$}
		\State $x_K \gets $ center of $K$
		\State $\psi_k \gets \varphi_{j,K} e^{i\omega p \cdot (x-x_K)}$
		\State $k \gets k+1$
	\EndFor
	\EndFor
	\State $N \gets k$ 
	\For{$i,j = 1,\ldots,N$}
		\State $\mathbf{M}_{ij} \gets \int_{\Omega} \psi_j(x) \conj{\psi}_i(x) \, dx$
		\State $\mathbf{M}^c_{ij} \gets \int_{\Omega} \frac{1}{c(x)^2} \psi_j(x) \conj{\psi}_i(x) \, dx$
		\State $\mathbf{A}_{ij} \gets a^{\gamma}_h(\psi_j, \psi_i)$
	\EndFor
	\State $\mathbf{u}^0,\mathbf{u}^1 \gets \mathbf{M}^{-1} (\int_{\Omega} u^0 \psi_j \,dx)_{j=1}^{N},
			\mathbf{M}^{-1} (\int_{\Omega} u^1 \psi_j \,dx)_{j=1}^{N}$ \Comment{Projection of initial conditions}
	\For{$t \gets 2\Delta t,\, 3\Delta t,\,\ldots,\, T$}
		\State $\mathbf{u}^0,\mathbf{u}^1 \gets \mathbf{u}^1,
			2\mathbf{u}^1-\mathbf{u}^0-\Delta t^2 (\mathbf{M}^c)^{-1} \mathbf{A} \mathbf{u}^1$ \Comment{Time marching}
	\EndFor
	\State $u_h \gets \sum_{j=1}^N \mathbf{u}^1_j \psi_j$										
\EndProcedure
\end{algorithmic}
\end{algorithm}

We note that it is very time consuming for
computing the mass and the stiffness matrices: $$\mathbf{M}^c:= \left( \int_{\Omega} c^{-2}\,\psi_j(x) \conj{\phi}_i(x) \right) \quad\mbox{ and }\quad \mathbf{A} := \Big(a^{\gamma}_h(\psi_j, \psi_i) \Big).$$
Instead of using numerical quadratures, an alternative method is to use the following formula given by \cite{bakhvalov1968evaluation}:
\begin{eqnarray}
\int_{-1}^{1} P_k(x) e^{i\omega x}\, dx = i^k\left(\frac{2\pi}{\omega}\right)^{1/2} J_{k+\frac{1}{2}}(\omega), \label{intlegendre}
\end{eqnarray}
where $J$ is the Bessel function of the first kind and $P_k(x)$ is the Legendre polynomial of degree $k$.
By writing the integrands into a product of functions in $x$ and in $y$, all the nonzero entries of $\mathbf{A}$ can be computed by equation \eqref{intlegendre}.
For the matrix $\mathbf{M}^c$, if the mesh size $h$ is small enough such that $\rho(x,y) := 1/c^2(x,y)$ is approximately given by
$$
\rho|_{K} \approx \sum_{r} \sum_{s} \rho_{rs} P_r(x) P_s(y),
$$
then we can approximate the nonzero entries using \eqref{intlegendre} with an error estimate of the form
{
$$
\abs{ \int g(x) e^{i\omega x} dx - \int\sum {a_k P_k(x) e^{i\omega x} dx } } \leq \norm{g - \sum a_k P_k }_1.
$$}

\subsection{Online/offline scheme}
Suppose we need to solve the problem \eqref{acoustic} with a fixed velocity profile but many initial conditions.
From the constructions above, we see that the set of basis functions depends on the pre-defined sets $\lbrace \Thetadef \rbrace$,
which depend on the initial conditions (c.f. Section \ref{sec:separation}).
It is obvious that one does not want to compute the mass and the stiffness matrices for each initial condition. 
In this section, we propose an offline/online scheme to 
address this issue by pre-computing the entries of one mass and one stiffness matrix that 
can be used for all choices of initial conditions. 
That is, these mass and stiffness matrices are independent of the initial conditions. 

First of all, we define a predefined set of ray directions $\Thetapre$.
We assume that this set contains a large set of ray directions that are needed to compute the solution
for all initial conditions with good accuracy. 
Note that this set can be considered as a snapshot space using the terminology in \cite{gao2015generalized}. 
It is obvious that one can take $\Thetapre = \mathbb{R}^2 \backslash \{0\}$, but this choice is not good in practice
as it is too large and too expensive to use in the offline stage. Instead, 
we put a restriction on our solver such that it only handles the case that the ray directions lie in an annulus
$R = \lbrace p \in \mathbb{R}^2 : 0 < p_1 < p < p_2 \rbrace$.
Then we choose the set of predefined phases that satisfies
$R \subset	\Thetapre + B_{\delta/2}(0)$, for some small $\delta > 0$.
Next, we present the offline and the online stages. 
In the offline stage, we compute the matrices $\mathbf{M^c}$ and $\mathbf{A}$ in Algorithm 6 using the ray directions in $\Thetapre$.
We note that this computation is performed before the actual simulations and is independent of the initial conditions. 
We also note that the cost of this offline computation depends on the size of the set $\Thetapre$.

In the online stage, when the initial condition is given, 
we will apply a modification of the phase-construction procedure as in Algorithm 5 to obtain the desired ray directions. 
The modification is discussed as follows.
The main idea is that, we will construct the dominant ray directions as before,
but we will use the ray directions in the set $\Thetapre$ that are closest to those dominant ray directions as the basis functions. 
The advantage is that we can extract sub-matrices of $\mathbf{M^c}$ and $\mathbf{A}$ from the offline stage
for the simulations. 
Next, we present in detail how this is performed. 
Firstly, we apply the original phase separation step with $\varepsilon$ replaced by $\varepsilon + 2\delta$.
$$
\widetilde{\Theta}^{**}_{K} \gets \textsc{PhaseSep}(\Thetadef, \widetilde{\Theta}^*_{K}, \varepsilon+2\delta).
$$
Secondly, for each ray direction in $\widetilde{\Theta}^{**}_{K}$, we choose the closest ray direction in $\Thetapre \cup \Thetadef$ to form the set $\widetilde{\Theta}_K$. More precisely we take 
$$
\widetilde{\Theta}_{K} \gets \Thetadef \cup \left\lbrace
\underset{p \in \Thetapre \cup \Thetadef}{\mbox{arg\,min}} d(p,q): q \in \widetilde{\Theta}^{**}_{K} \right\rbrace
$$
Given that $d(\widetilde{\Theta}^{**}_{K}, \Thetapre) < \delta/2$, the set $\widetilde{\Theta}_{K}$ is an $(\varepsilon/2)$-separable set such that
$d(\widetilde{\Theta}_{K}, \widetilde{\Theta}^*_{K}) < \varepsilon + 3\delta$.
After we obtain this new $\widetilde{\Theta}_{K}$, we retrieve the corresponding entries from the precomputed $\mathbf{M}^c$ and $\mathbf{A}$. Besides,
we only need to compute the entries corresponding to $\Thetadef \setminus \Thetapre$ in the online stage.
Afterwards, we solve the IPDG scheme as in Algorithm 6.

\subsection{Remarks on computation issues}
\subsubsection{Parallel computation for the ray-construction stage}
Suppose we have multiple processing units available for simultaneous calculations.
We note that the ray direction and position of a ray do not depend on other rays but only the velocity field.
In other words, the computation for the rays can be carried out with duplicated velocity field on each processing unit.
Suppose we split a discrete wavefront $W^n$ into smaller wavefronts (with overlapping endpoints), solve the ray equations for each smaller wavefront, apply the reconstruction \textsc{WavefrontRecon}, and combine the reconstructed wavefronts into one wavefront, the resulting wavefront is the same as the discrete wavefront $W^{n+1}$ describe in section 3.2.
Therefore if we modify Algorithm 5 and perform the procedures \textsc{WavefrontPropagate}, \textsc{WavefrontRecon} and \textsc{PhaseDet}
for each smaller wavefronts on different processing units, the resulting set of phases
$\widetilde{\Theta}_K$ would be the same after performing the \textsc{PhaseSep} procedure.
The speedup of this parallel algorithm depends on the maximum number of rays on each wavefront throughout the simulation,
which may not be clear before the computation.
How to choose the initial splitting of the wavefronts remains a question.

\subsubsection{Conditioning of the IPDG system}
\label{sec:pod}
In this part, we consider the conditioning of the mass matrix. 
As an illustration, we consider $K = [0,h]^2$, $\widetilde{\Theta}_K = \lbrace (1,0)^T, (\cos\theta, \sin\theta)^T\rbrace$, and the local approximation space $V(\widetilde{\Theta}_K)$.
In Figure 3, for each mesh size $h$, we compute the condition number of the mass matrix corresponding to the basis $\varphi_{j,K} e^{i p \cdot (x-x_0)}$, where $\varphi_{j,K}$ is a standard bilinear basis, $p \in \widetilde{\Theta}_K$, and $x_0$ is the center of $K$.
The figure suggests that the resulting system is very ill-conditioned for small $h$ and $\theta\approx 0$.
\begin{figure}[h]
\centering
\includegraphics[width=7cm]{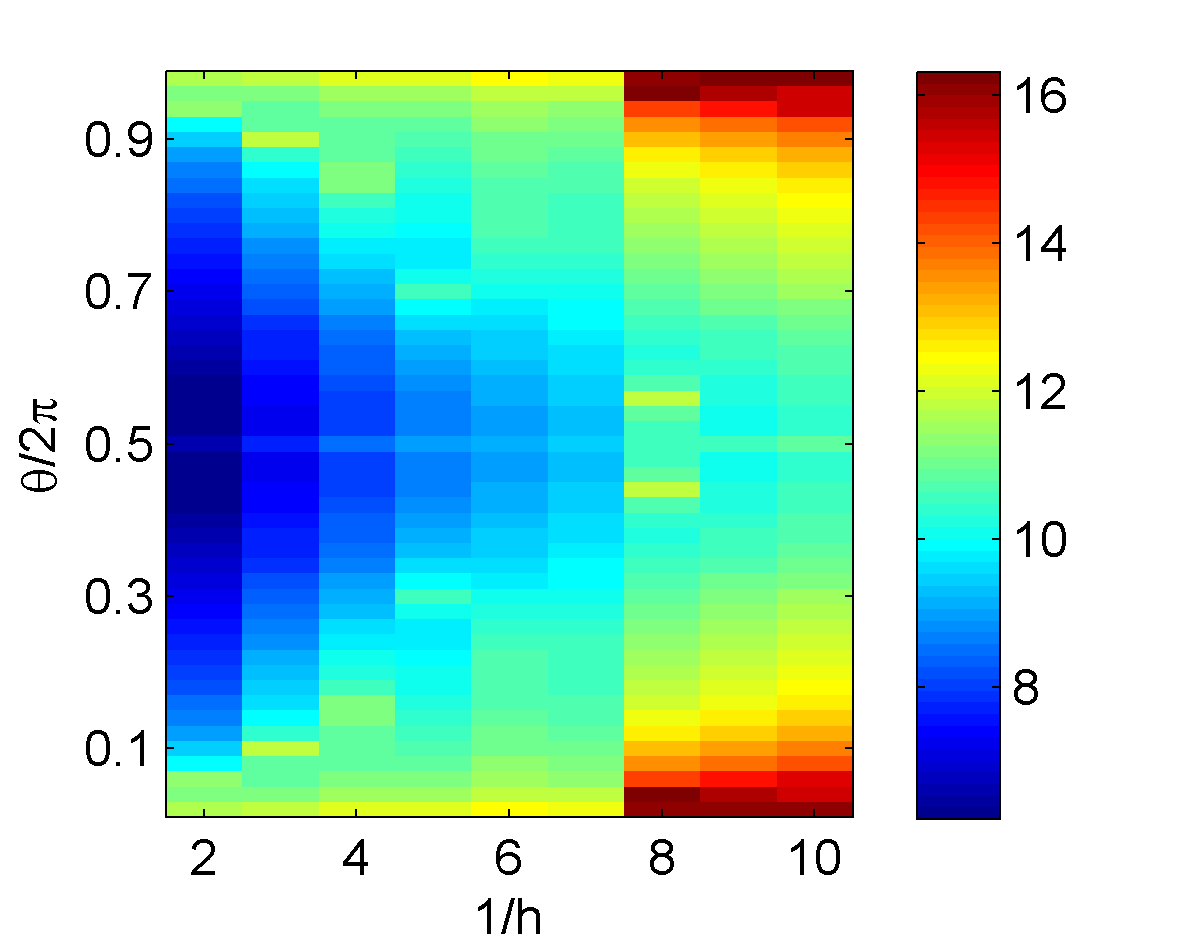}
\caption{
This figure shows the condition number (in $\log_{10}$) of the mass matrix corresponding to the phases $(1,0)^T$ and $(\cos\theta, \sin\theta)^T$ on a tile with side length $h$.}
\end{figure}

To form a better conditioned system, in addition to choosing a larger $\varepsilon$,
we also consider the proper orthogonal decomposition (POD) basis for the IPDG system \cite{kunisch2001galerkin, kunisch2002galerkin}.
The POD basis is obtained by considering a truncated singular value decomposition of the matrix
$\mathbf{M}^c = \left( \int \frac{1}{c^2} \psi_i \psi_j\, dx \right)_{ij} $, where the $\psi$'s are the basis in $V(\widetilde{\Theta}_K)$.
Suppose the eigenvalues of $\mathbf{M}^c$ are given by 
$\lambda_1 \geq \lambda_2 \ldots \geq \lambda_N > 0$.
For a fixed $\eta$, we choose the minimum $N^*$ such that
$$\sum_{k=1}^{N^*} \lambda_k \geq 1-\eta.$$
Then the POD basis is given by the eigenvectors corresponding to $\lambda_1,\ldots, \lambda_{N^*}$.

\section{Numerical experiments}
\label{sec:num}
In this section, we present some numerical examples to show the performance of our proposed method.
For all the cases, the computational domain is $\Omega := [0,1]^2$, and $\Omega$ is subdivided into $N\times N$ square cells.
We define the mesh size $h$ to be $\frac{1}{N}$.
We simulate the equation \eqref{acoustic} with non-constant wave speed $c(x)$ for the times $0\leq t \leq 1 := T$
using the time step $\Delta t$.
The equation is supplemented with the periodic boundary condition and the initial condition
will be specified for each of the cases considered below. 
In all of our examples, we consider two different media with speeds given by
$$
c_1(x) = 1 +\frac{1}{5}\exp(-150+300(x_1+x_2)-240x_1x_2-180(x_1^2+x_2^2)),
$$
$$
c_2(x) = \frac{1}{5} \sin( 4\pi x_2 )+1.
$$
We illustrate these wave speed profiles in Figure~\ref{egwavespeed}.
\begin{figure}[h]
\centering
\includegraphics[width = 7cm]{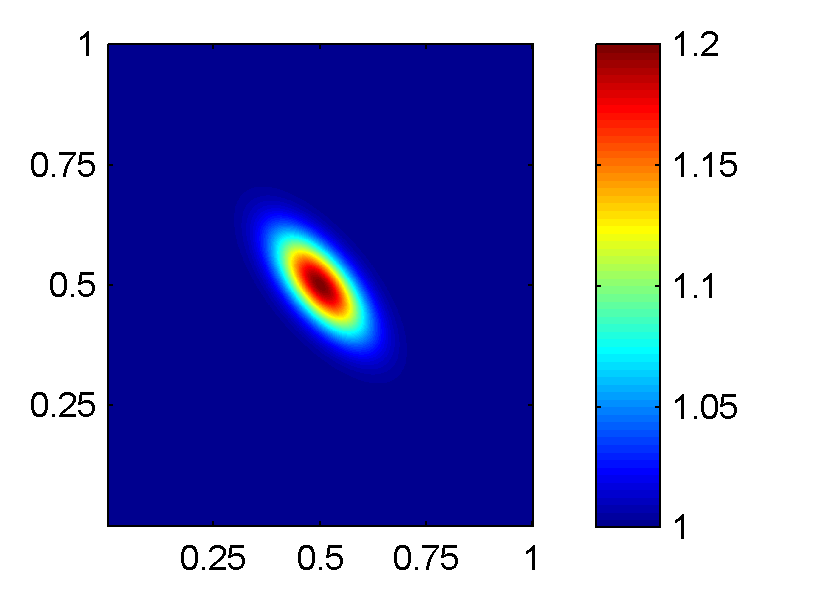}
\includegraphics[width = 7cm]{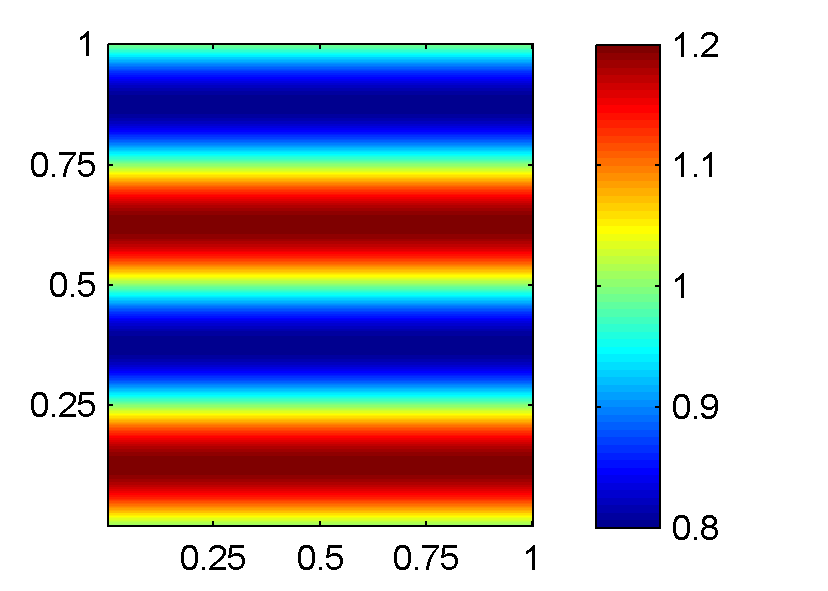}
\caption{The wave speed of the media. Left: $c_1$. Right: $c_2$. }
\label{egwavespeed}
\end{figure}
To benchmark the performance of our method, we will compare the numerical solutions computed by our method to a reference solution computed by a pseudospectral scheme
for a finer mesh. 

\subsection{Example 1}
In this example, we test our method without performing the POD described in Section \ref{sec:pod}
and compare it to the standard IPDG method (with polynomial basis) as $\omega\to \infty$ but $\omega h = \bigO{1}$.
We consider the wave speed $c_1$.
For the initial conditions, we take $L=1$, the phase function $\phi_1(x) = x_1$
and the amplitude functions $A_1(x) = 1$ and $B_1(x,\omega) = -i\omega$. 
For the numerical solutions $u_h^0$ and $u_h^1$ at the first two time steps,
we can compute them by the given initial conditions and a forward Euler scheme in time respectively. 


Since $\nabla \phi_1(x) = (1,0)^T$ for all values of $x$,
in the phase-construction stage in Algorithm 5, we can take $\Thetadef = \lbrace(1,0)^T\rbrace$
for all cells $K$ in the partition of the domain. 
For the wavefronts, we will use $9$ level set functions defined by $\phi(x) = \beta$
with $\beta = 0.1, 0.2,\cdots, 0.9$.
In addition, we take the tolerance function $\mbox{tol}(x,p) = 10 \abs{x} + 100 \abs{p}$ and the separation parameter $\varepsilon = 0.2$.
For both solving the equations (\ref{eq:rayODE}) and 
the time marching stage in Algorithm 6, we take the time step $\Delta t = h/100$. Finally, we take the penalty parameter $\gamma = 10$.

We consider the numerical solution of our scheme using different mesh sizes $h$ and $\omega$ with $\omega h = \pi$. We list the relative errors in Table \ref{eg1tblerr}.
We also consider the numerical solutions given by standard IPDG scheme with the same time step size $\Delta t = h/100$ and penalty parameter $\gamma = 10$.
We take more degrees of freedoms and let $\omega h = \pi/10$. The numerical results are shown in Table~\ref{eg1tblpoly}.
\begin{table}[h]
\centering
\begin{tabular}{|c|c|c|c|c|c|}
\hline
$\omega$ & $1/h$ & $\omega h$ & \begin{tabular}{c}
Condition\\number \end{tabular}
& \begin{tabular}{c}Number of\\ dof\end{tabular}
& \begin{tabular}{c}Relative\\$L^2$ error (\%) \end{tabular} \\ \hline
$10\pi$ & 10 & \multirow{5}{*}{$\pi$} & 7.59e+07& 456 & 16.856 \\
$20\pi$& 20 & & 8.48e+07&1844& 14.339 \\
$40\pi$& 40 & & 1.08e+08&7376& 10.976 \\
$80\pi$& 80 & & 1.16e+08&29532& 10.128 \\
$160\pi$& 160 & & 1.33e+08&118184& 11.183 \\\hline
\end{tabular}
\caption{ In Example 1, the numerical results corresponding to our method.}
\label{eg1tblerr}
\end{table}
\begin{table}[h]
\centering
\begin{tabular}{|c|c|c|c|c|}
\hline
$\omega$ & $1/h$ & $\omega h$ & \begin{tabular}{cc}Number\\of dof\end{tabular} & \begin{tabular}{cc}Relative\\$L^2$ error (\%)\end{tabular} \\ \hline
$10\pi$ & 100 & \multirow{5}{*}{$\pi/10$} & 40000 & 20.11 \\
$20\pi$& 200 & & 160000 & 32.72 \\
$40\pi$& 400 & & 640000 & 57.83 \\
$80\pi$& 800 & & 2560000 & 105.00 \\
$160\pi$& 1600 & & 10240000 & 176.80 \\\hline
\end{tabular}
\caption{
In Example 1, the numerical results corresponding to the standard IPDG method with bilinear basis.}
\label{eg1tblpoly}
\end{table}

In these numerical results, we see that the relative $L^2$-errors of
the numerical solutions for our method
do not increase significantly as the frequency $\omega$ increases,
in contrast to the numerical solutions of the standard IPDG method with bilinear basis.
We also see that the condition numbers of the mass matrices without using POD are relatively large. 

In Figure~\ref{eg1ndof}, we show the number of ray directions we captured at each observation point $x_K$.
The number of rays in different regions of the domain suggests that
the phase-construction stage is able to capture the approximate ray directions in a consistent manner.
Also, there are maximum 2 phases at each observation point or 8 degrees of freedom in a tile.
Therefore, there are $2\times\sqrt{8} \approx 5.66$  points per wavelength for the numerical solutions of our method.

We also demonstrate the quality of the solution by comparing our solution to the reference solutions in Figure~\ref{eg1ps} and \ref{eg1dg}.
In Figure~\ref{eg1diff}, we compute the absolute difference between our solutions and the reference solutions.
In Figure~\ref{eg1waveform}, we compare the real parts of our solutions to the reference solutions at $T = 1$ and $x_2 = 0.3$.

\begin{figure}[p]
\centering
\includegraphics[width=5cm]{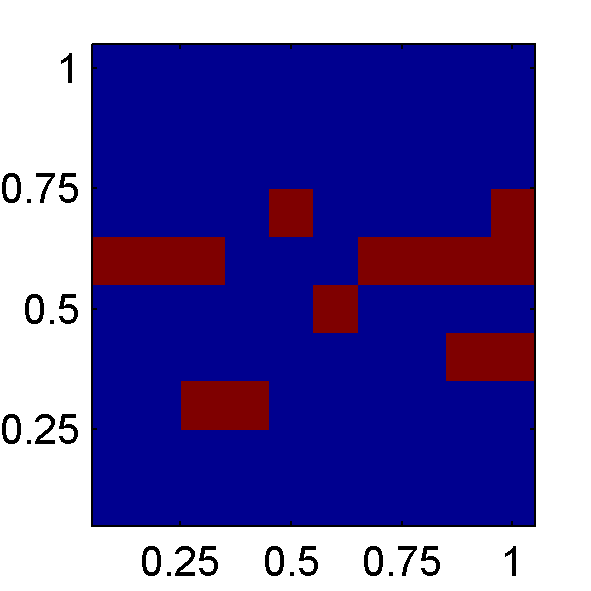}
\includegraphics[width=5cm]{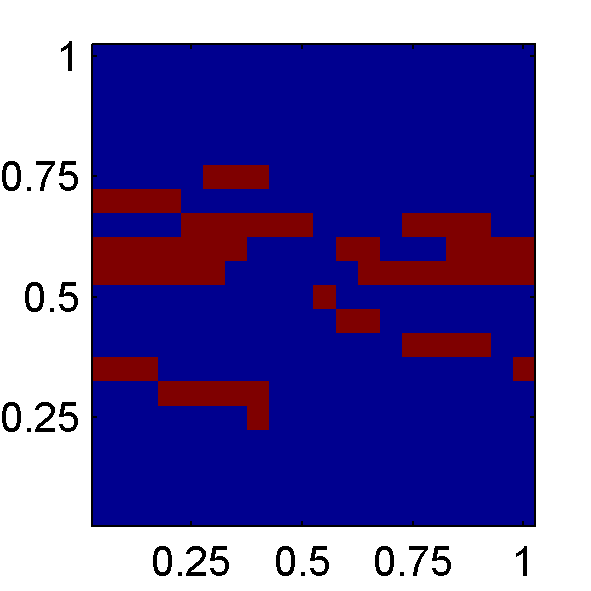}
\includegraphics[width=5cm]{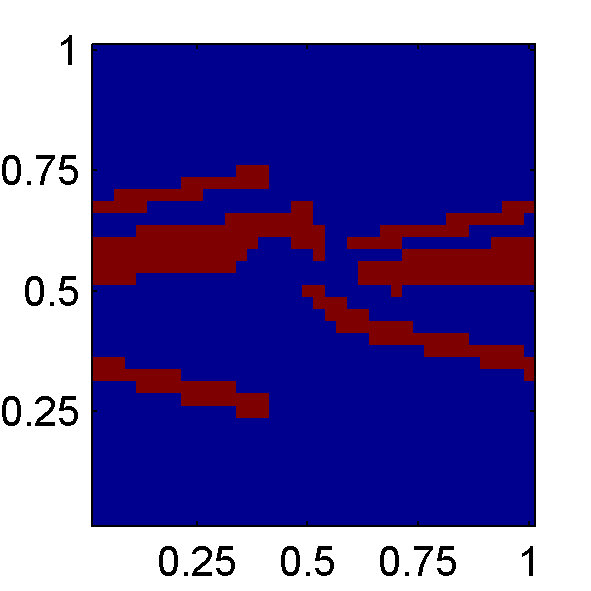}
\includegraphics[width=5cm]{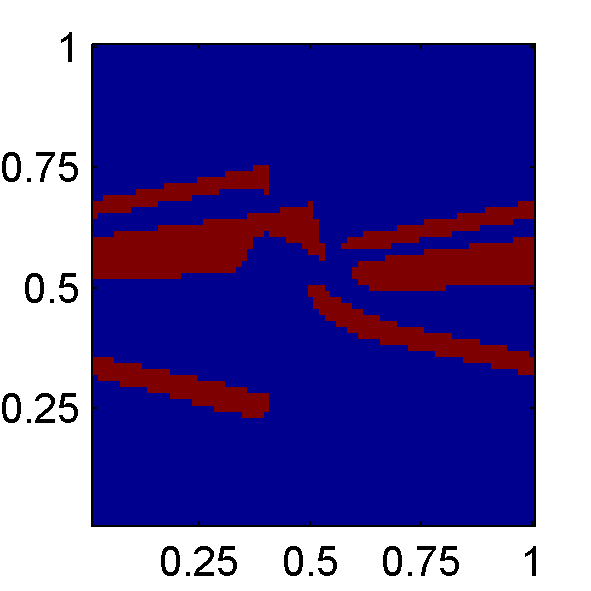}
\includegraphics[width=5cm]{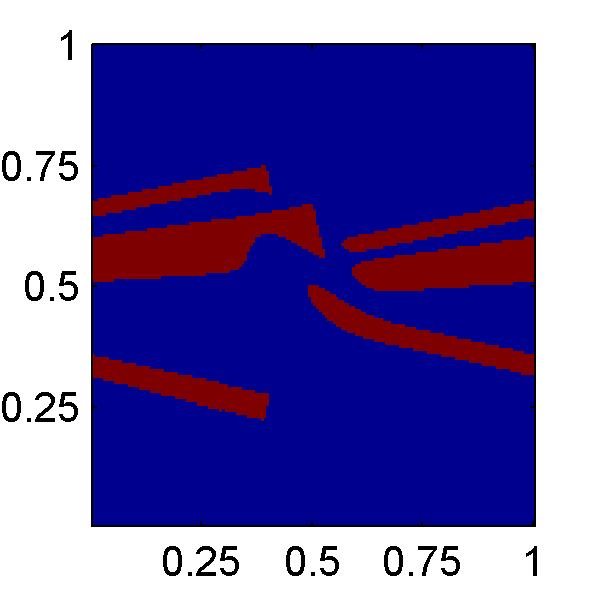}
\caption{
In Example 1, the number of phases obtained in the $\textsc{PhaseSep}$ procedure for each tile $K$, i.e. $\abs{\widetilde{\Theta}_{K}}$.
These figures are corresponding to $1/h = 10, 20, 40, 80$ and $120$, respectively.
Each tile in these figures represent the number of phase (Blue: 1, Red: 2).
The horizontal and vertical axis represent the $x_1$ and $x_2$ coordinates.}
\label{eg1ndof}
\end{figure}
\begin{figure}[p]
\centering
\includegraphics[height=5cm]{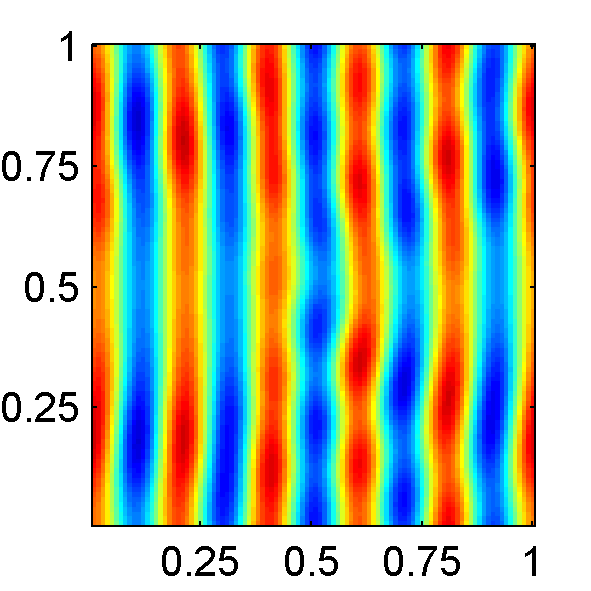}
\includegraphics[height=5cm]{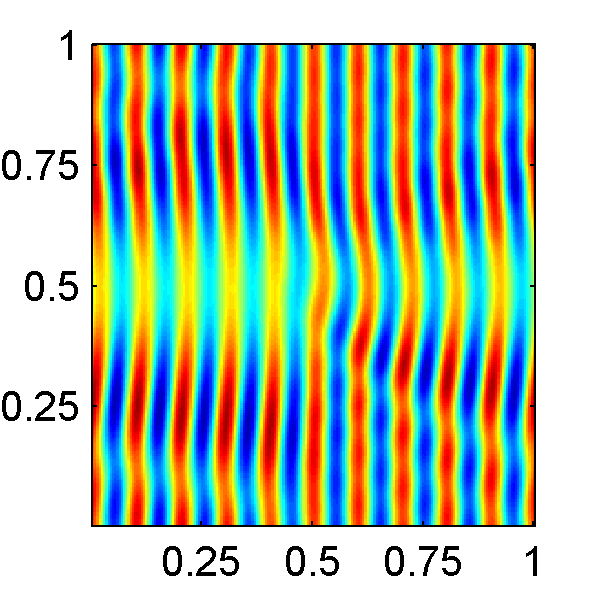}
\includegraphics[height=5cm]{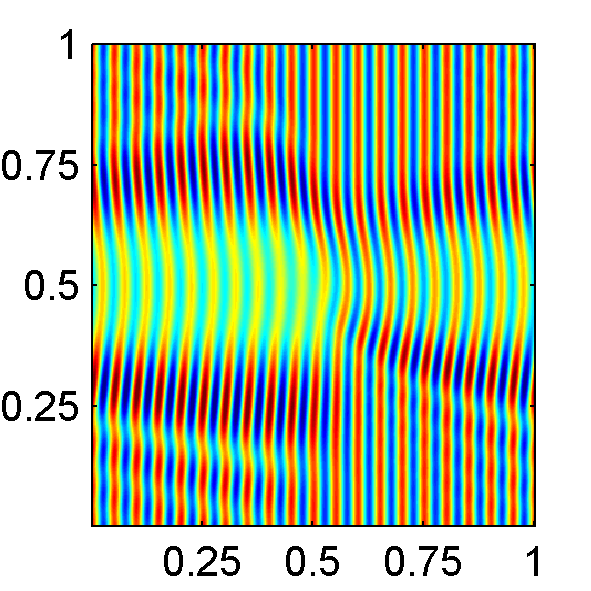}
\includegraphics[height=5cm]{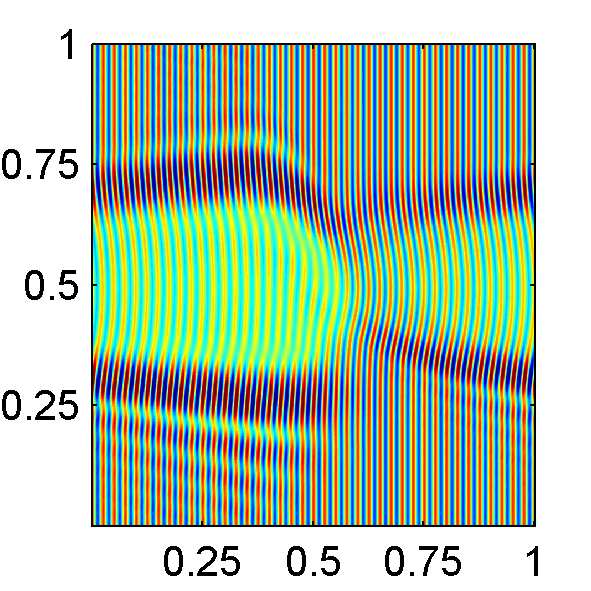}
\includegraphics[height=5cm]{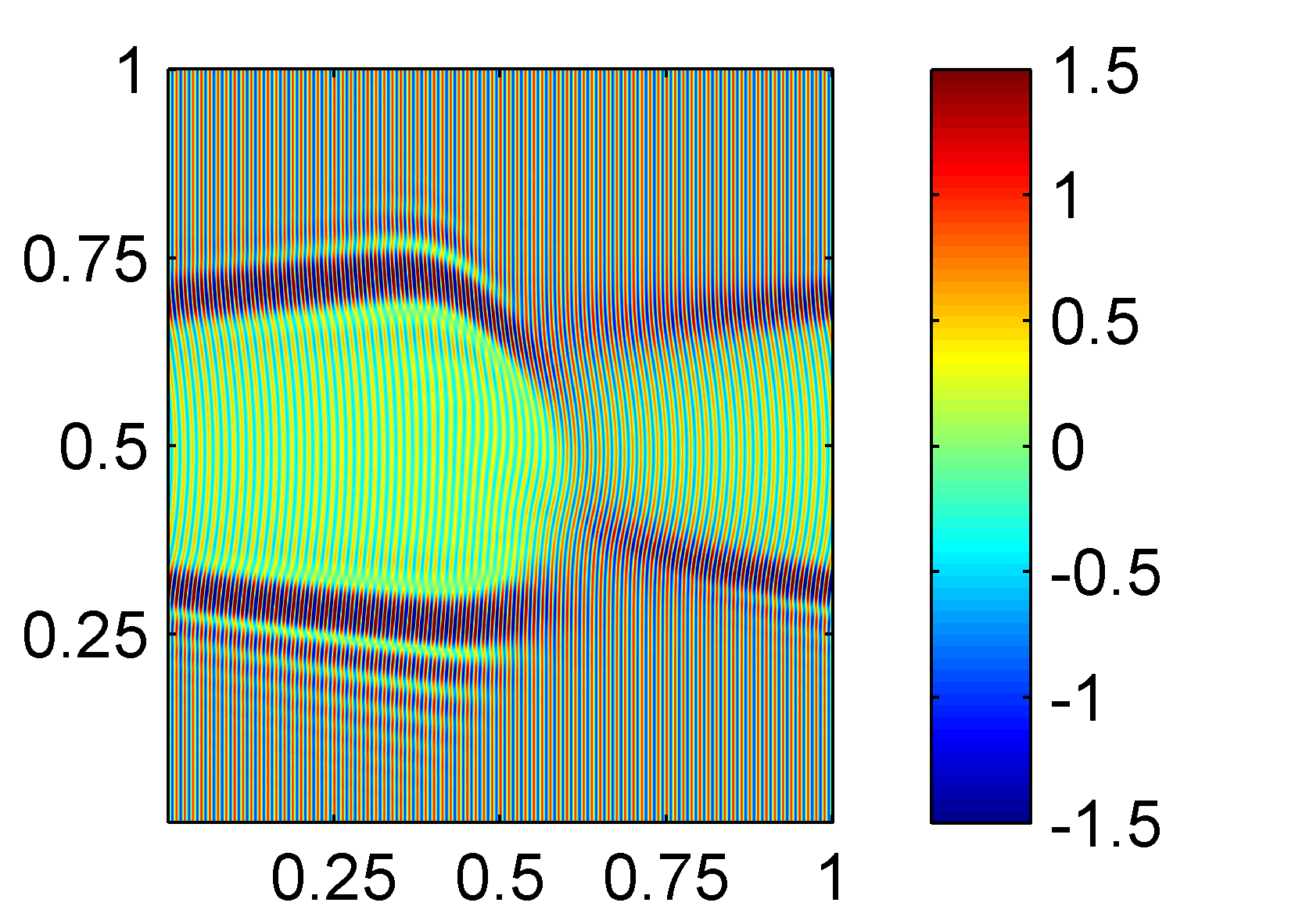}
\caption{ In Example 1, the real part of the reference solutions for $1/h = 10, 20, 40, 80$ and $120$, respectively. The horizontal and vertical axis represent the $x_1$ and $x_2$ coordinates. }
\label{eg1ps}
\end{figure}
\begin{figure}[p]
\centering
\includegraphics[height=5cm]{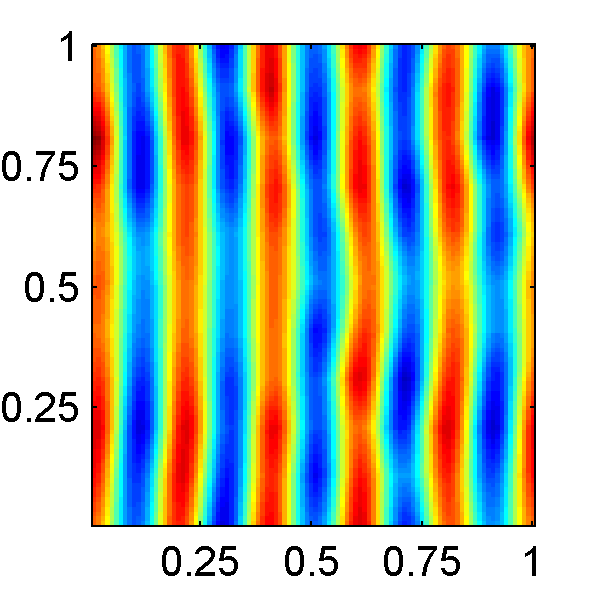}
\includegraphics[height=5cm]{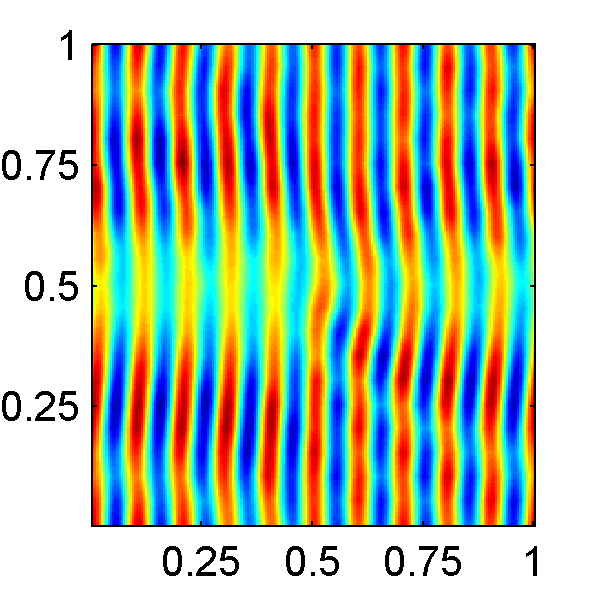}
\includegraphics[height=5cm]{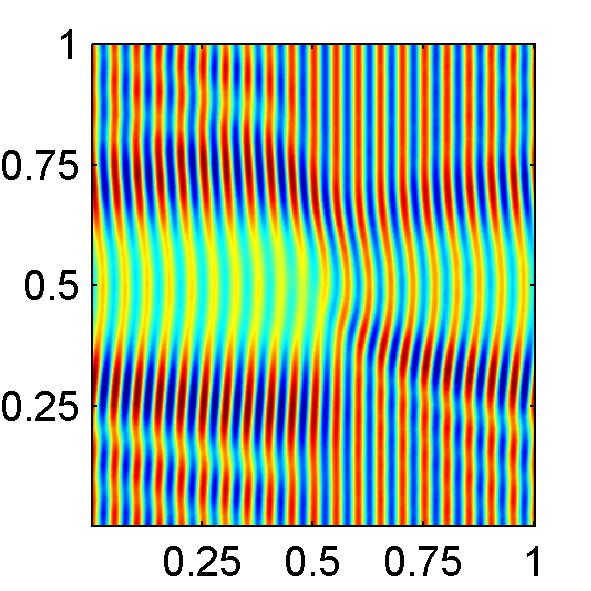}
\includegraphics[height=5cm]{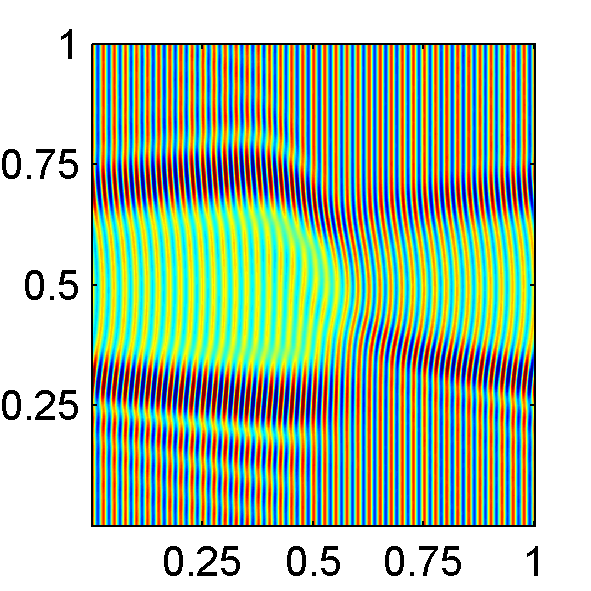}
\includegraphics[height=5cm]{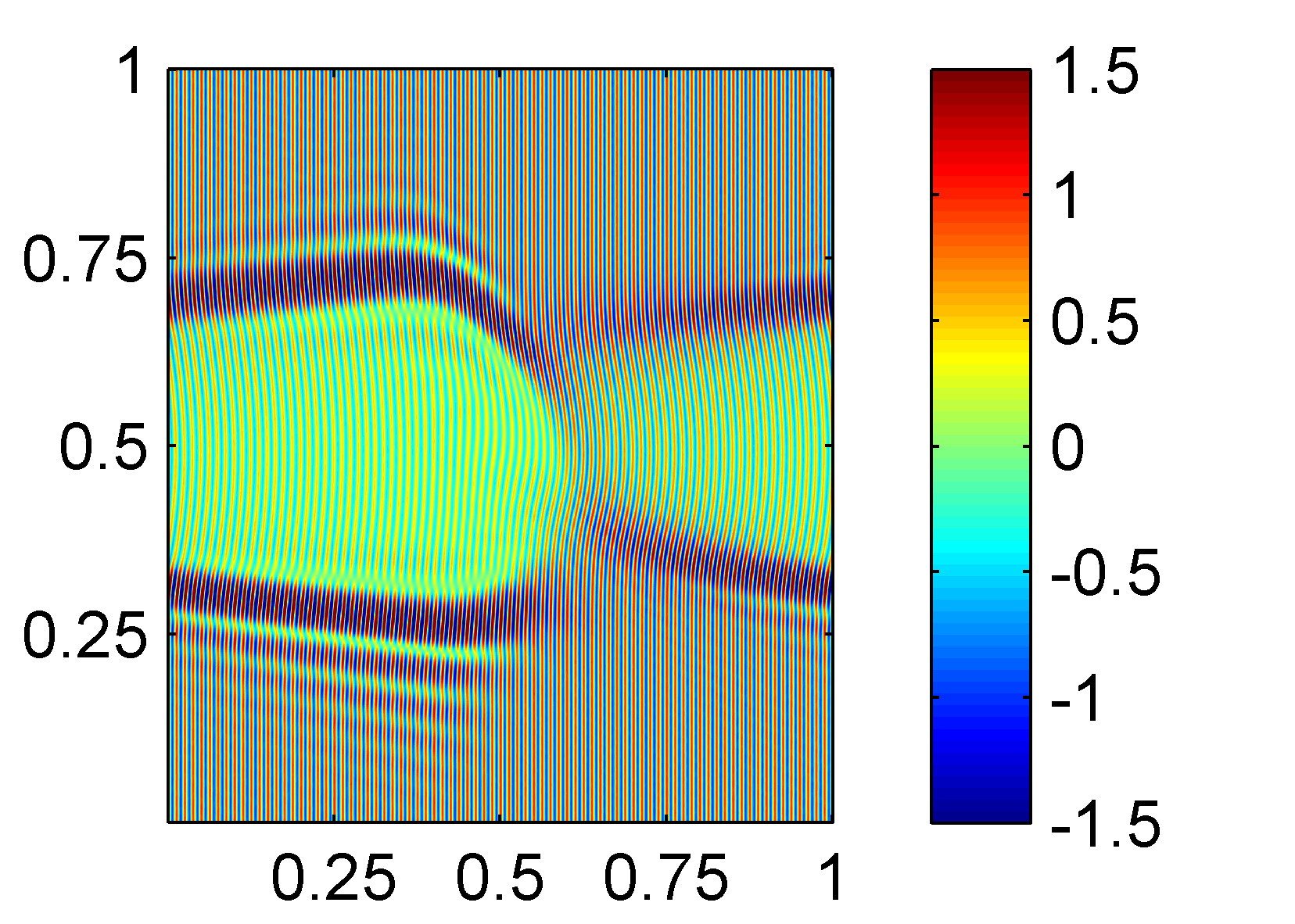}
\caption{ 
In Example 1, the
real part of the solutions obtained by our method for $1/h = 10, 20, 40, 80$ and $120$, respectively. The horizontal and vertical axis represent the $x_1$ and $x_2$ coordinates.}
\label{eg1dg}
\end{figure}
\begin{figure}[p]
\centering
\includegraphics[height=5cm]{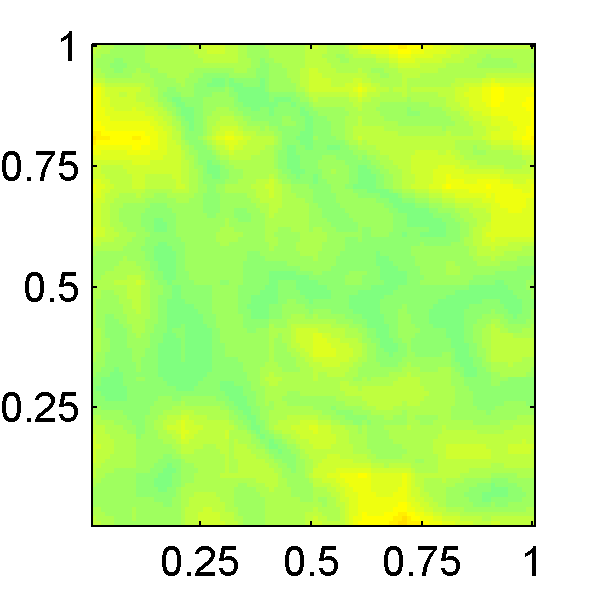}
\includegraphics[height=5cm]{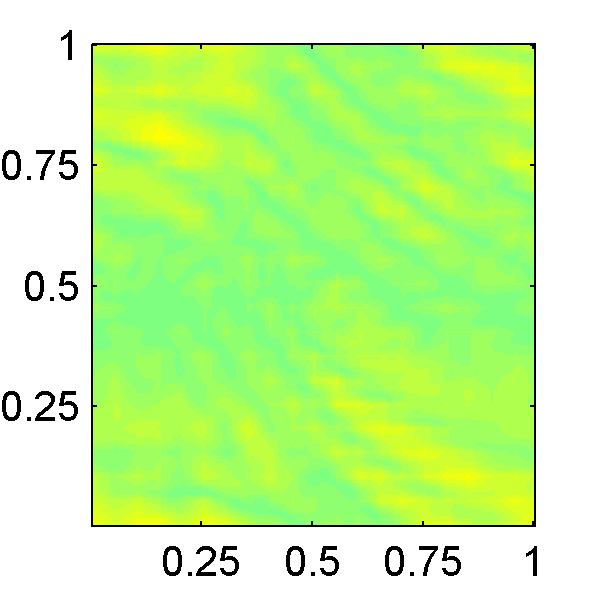}
\includegraphics[height=5cm]{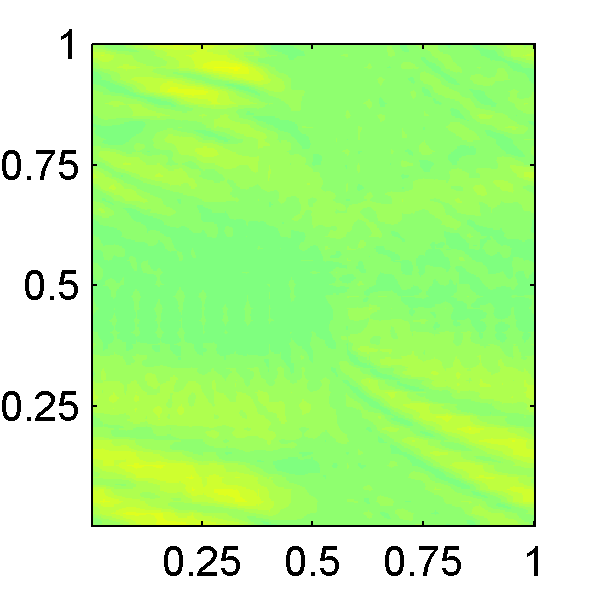}
\includegraphics[height=5cm]{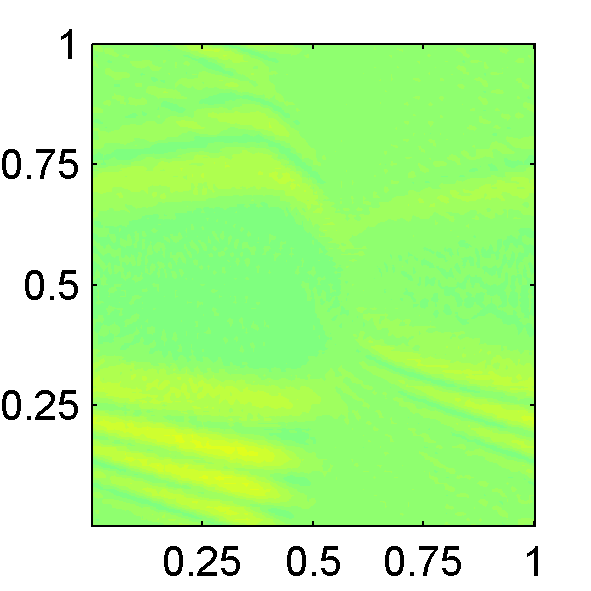}
\includegraphics[height=5cm]{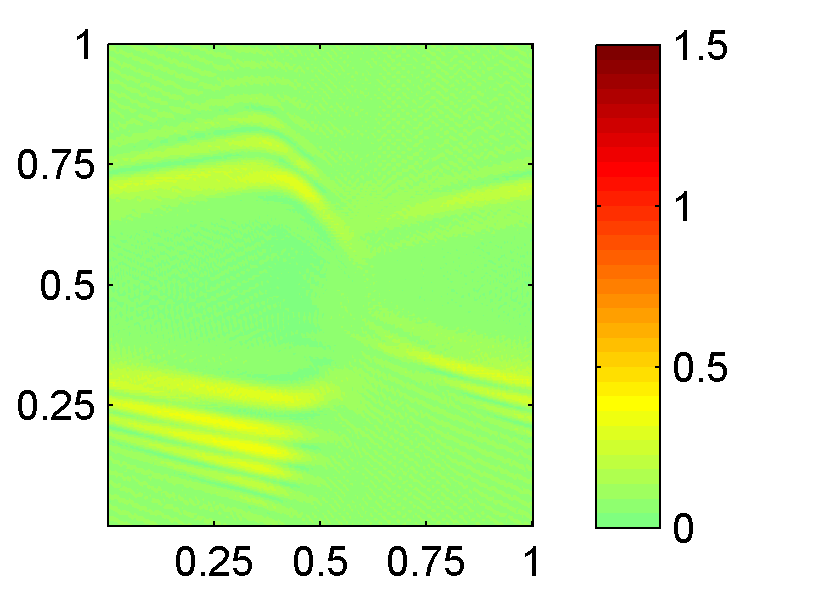}
\caption{In Example 1, the
absolute difference between the reference solution and the solutions obtained from our method for $1/h = 10, 20, 40, 80$ and $120$, respectively. The horizontal and vertical axis represent the $x_1$ and $x_2$ coordinates.}
\label{eg1diff}
\end{figure}
\begin{figure}[p]
\centering
\includegraphics[height=5cm]{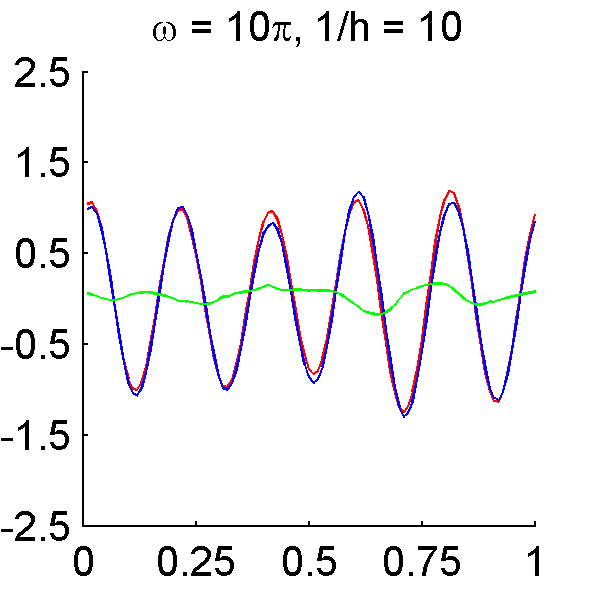}
\includegraphics[height=5cm]{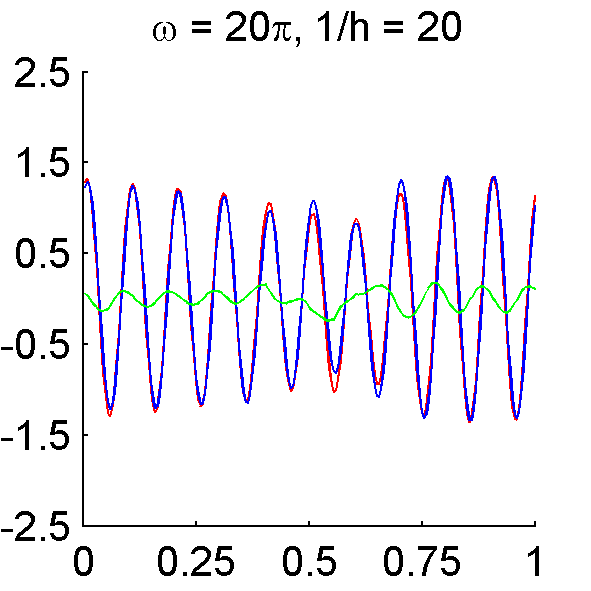}
\includegraphics[height=5cm]{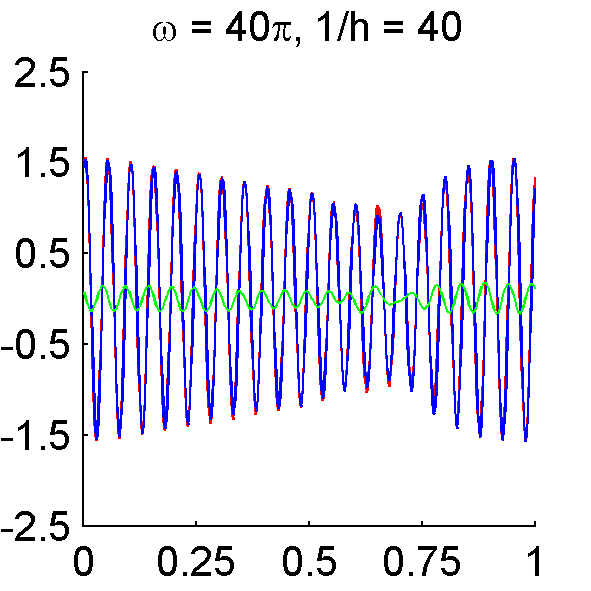}
\includegraphics[height=5cm]{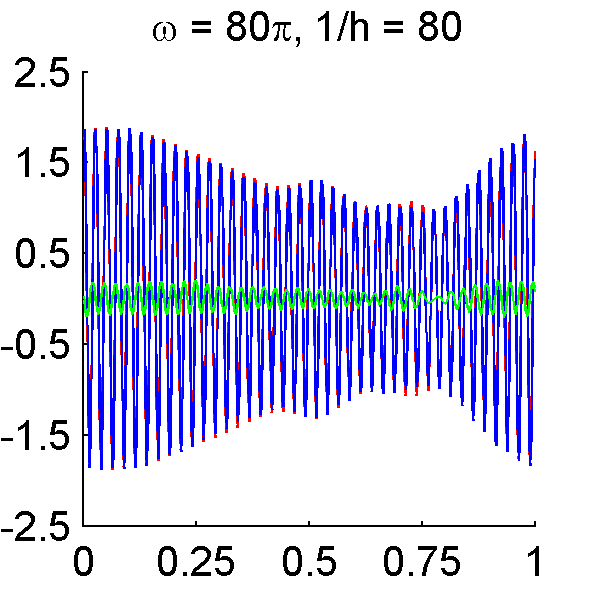}
\includegraphics[height=5cm]{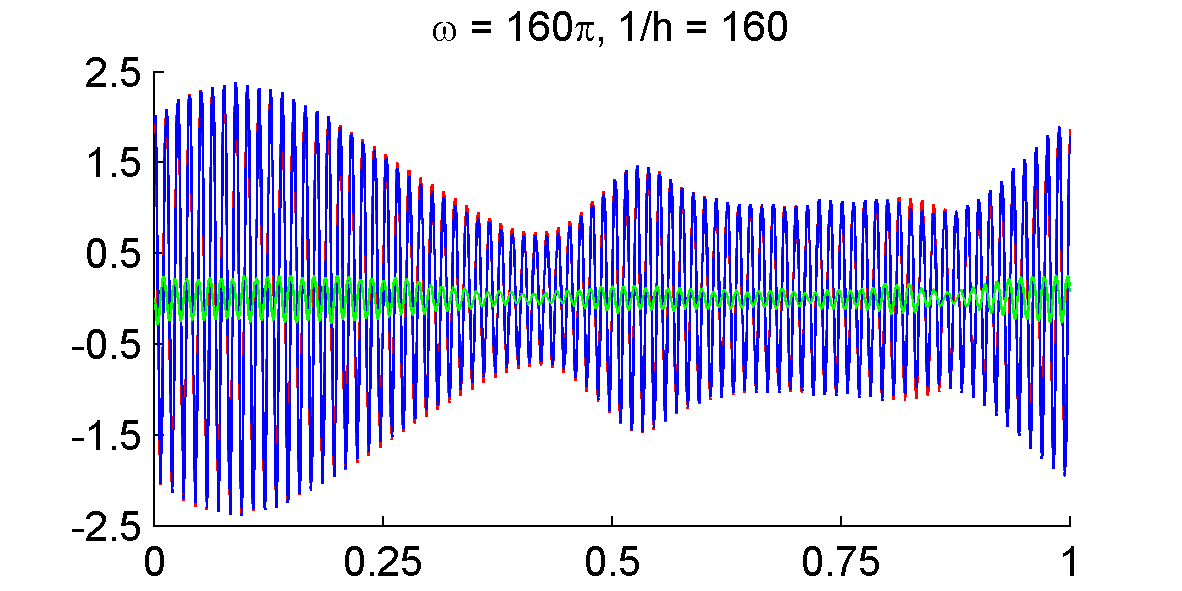}
\caption{
In Example 1, the real part of solutions obtained from our method (blue),
the reference solutions (red) and the difference between these solutions (green)
at $T=1$ and $x_2 = 0.3$
under different settings of mesh size $h$ and the parameter $\omega$.
The horizontal axis represent the $x_1$ coordinate.}
\label{eg1waveform}
\end{figure}
\subsection{Example 2}
In this example, we test the $h$-convergence of our method with the use POD basis presented in Section \ref{sec:pod}.
We consider the same wave speed, initial condition, initial wavefronts and parameters $\Thetadef, \Delta t, \mbox{tol}, \varepsilon, \gamma$ as in Example 1. Instead of solving the ill-conditioned system as $h\to 0$, we consider the POD system with parameter $\eta = 10^{-7}$.

In Table \ref{eg2tblerr} we list the relative errors of the approximate solutions.
We also compare the condition numbers and degrees of freedom of the POD systems to those of the original systems.
We observe that the POD systems has a better conditioning with number of degrees of freedom about the same with the original systems.
Our results suggests that the $h$-convergence of the POD systems is first-order.
Also, the relative $L^2$-error improves as the $\omega$ increases with fixed $\omega h$.

\begin{table}[h]
\centering
\begin{tabular}{|c|c|c|c|c|c|c|c|}
\hline
\multirow{2}{*}{$\omega$} &
\multirow{2}{*}{$1/h$} &
\multirow{2}{*}{$\omega h$} & 
\multicolumn{2}{|c|}{Condition number} &
\multicolumn{2}{|c|}{Number of dof} & Relative \\ \cline{4-7}
 &  &  & 
original & POD & original & POD
& $L^2$ error (\%)\\ \hline
\multirow{5}{*}{$10\pi$} & 10 & $\pi$ & 7.59e+07 & 4.95e+04 & 456 & 442 & 17.07\\
 & 20 & $\pi/2$ & 5.80e+09 & 6.87e+05 & 1844 & 1783 & 11.56 \\
 & 40 & $\pi/4$ & 3.52e+12 & 2.68e+06 & 7376& 7132 & 5.56 \\
 & 80 & $\pi/8$ & 6.89e+14 & 5.47e+06 & 29532& 28540 & 2.39 \\
 & 160 & $\pi/16$ & 5.65e+17 & 5.64e+06 & 118184& 113378 & 0.85\\\hline
 \multirow{4}{*}{$20\pi$} 
 & 20  & $\pi$    & 8.48e+07 & 5.86e+06 & 1844 & 1784 & 14.45\\
 & 40  & $\pi/2$  & 8.20e+09 & 6.71e+05 & 7376 & 7132 & 7.71\\
 & 80  & $\pi/4$  & 1.12e+12 & 2.70e+06 & 29532 & 28549 & 3.35\\
 & 160 & $\pi/8$  & 1.06e+14 & 5.59e+06 & 118184 & 114199 & 1.18\\\hline
 \multirow{3}{*}{$40\pi$} 
 & 40  & $\pi$    & 1.08e+08 &	5.81e+06 & 7376 & 7136 & 10.99\\
 & 80  & $\pi/2$  & 7.17e+09 & 6.76e+05 & 29532 & 28549 & 5.04\\
 & 160 & $\pi/4$  & 6.59e+11 & 2.92e+06 & 118184 & 114238 & 2.15 \\\hline
\end{tabular}
\caption{ In Example 2, the condition number and the number of degrees of freedom of the original system and the POD systems,
and the relative $L^2$ error under different settings of $\omega$ and meshsize $h$.
}
\label{eg2tblerr}
\end{table}
\subsection{Example 3}
In this example, we consider more phases in the initial condition and the 
solution. We will observe the behavior of our approximate solution as $\omega\to\infty$ while keeping $\omega h = \bigO{1}$, using the POD system.

We consider the wave speed $c_1$ in Example 1.
For the initial condition, we take $L=2$, $\phi_1(x) = x_1$, $A_1=1$, $B_1=-i\omega$,
$\phi_2(x) = x_2$, $A_2=1$ and $B_2 = -i\omega$.
For the phase-construction stage in Algorithm 5, since $\nabla\phi_1 = (1,0)^T$ and $\nabla\phi_2 = (0,1)^T$, 
we take $\Thetadef = \lbrace (1,0)^T, (0,1)^T\rbrace$ for every cell $K$ in the partition of the domain. 
For the wavefronts, we will use the level sets $\phi_1=\beta$ and $\phi_2=\beta$, where $\beta = 0.1, 0.2, \cdots, 0.9$.
All other parameters are taken the same as that for Example 2.

In Table \ref{eg3tblerr} we list the relative errors of the approximate solutions, and we see that our method is robust with respect to the frequency $\omega$.
We also compare the condition numbers and degrees of freedom of the POD systems to those of the original systems.
We observe that the POD systems has a better conditioning with number of degrees of freedom about the same with the original systems.
We see that the relative $L^2$-error of the numerical solution for the POD system does not increase significantly as the $\omega$ increase.

\begin{table}[h]
\centering
\begin{tabular}{|c|c|c|c|c|c|c|c|}
\hline
\multirow{2}{*}{$\omega$} &
\multirow{2}{*}{$1/h$} &
\multirow{2}{*}{$\omega h$} & 
\multicolumn{2}{|c|}{Condition number} &
\multicolumn{2}{|c|}{Number of dof} & Relative \\ \cline{4-7}
 &  &  & 
original & POD & original & POD
& $L^2$ error (\%)\\ \hline
$10 \pi$ & 10 & \multirow{5}{*}{$\pi$} & 1.12e+09 & 2.22e+06 & 912 & 884 & 15.09\\ 
$20 \pi$ & 20 & & 3.91e+08 & 4.23e+06 & 3688 & 3566 & 11.54 \\ 
$40 \pi$ & 40 & & 4.39e+09 & 4.43e+06 & 14752 & 14263 & 8.93 \\ 
$80 \pi$ & 80 & & 3.46e+10 & 4.56e+06 & 59064 & 57090 & 8.75 \\ 
$160 \pi$ & 160 &  & 8.07e+11 & 4.54e+06 & 236368 & 228443 & 9.75 \\ \hline
\end{tabular}
\caption{ In Example 3, the condition number and the number of degrees of freedom of the original system and the POD systems,
and the relative $L^2$ error under different settings of $\omega$ and meshsize $h$.}
\label{eg3tblerr}
\end{table}

In Figure~\ref{eg3ndof}, we show the number of phases we captured at each observation point $x_K$.
The number of rays in different region of the domain suggests that
the phase-construction stage capture the approximate phases in a consistent manner.
Also, there are maximum 4 phases at each observation point or 16 degrees of freedom in a tile.
Therefore there are $2\times\sqrt{16} = 8$ points per wavelength in the original systems of our method.

We also demonstrate the quality of the solution by comparing our solution to the reference solutions in Figure~\ref{eg3ps} and \ref{eg3dg}.
In Figure~\ref{eg3diff}, we compute the absolute difference between our solutions and the reference solutions.
In Figure~\ref{eg3waveform}, we compare the real parts of our solutions to the reference solutions at $T = 1$ and $x_2 = 0.3$.

\begin{figure}[p]
\centering
\includegraphics[width=5cm]{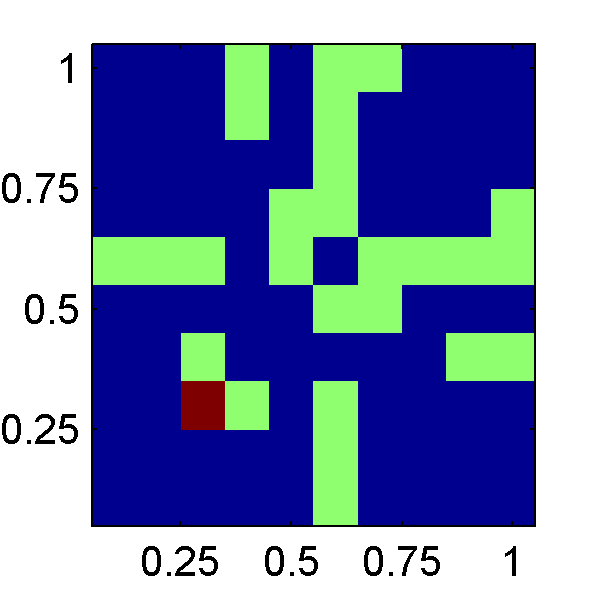}
\includegraphics[width=5cm]{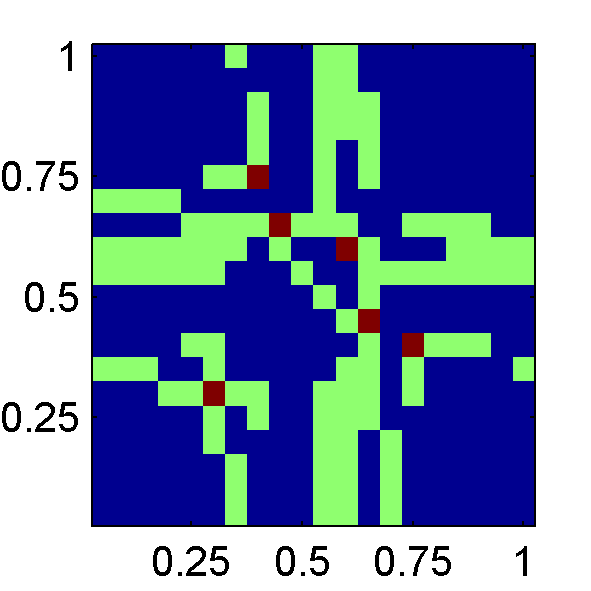}
\includegraphics[width=5cm]{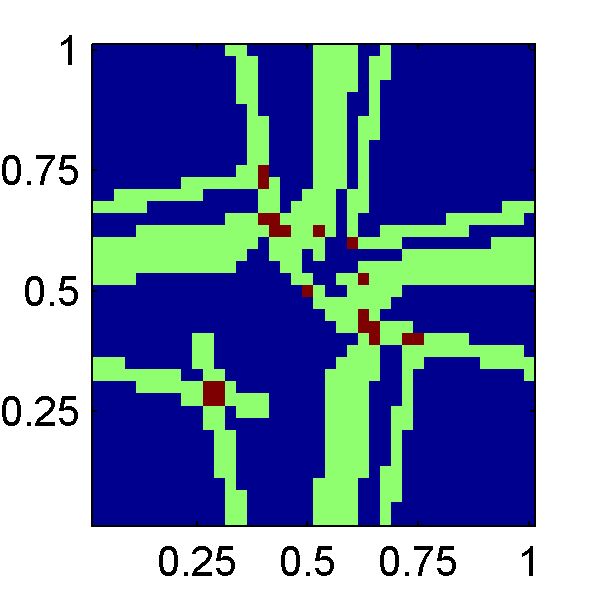}
\includegraphics[width=5cm]{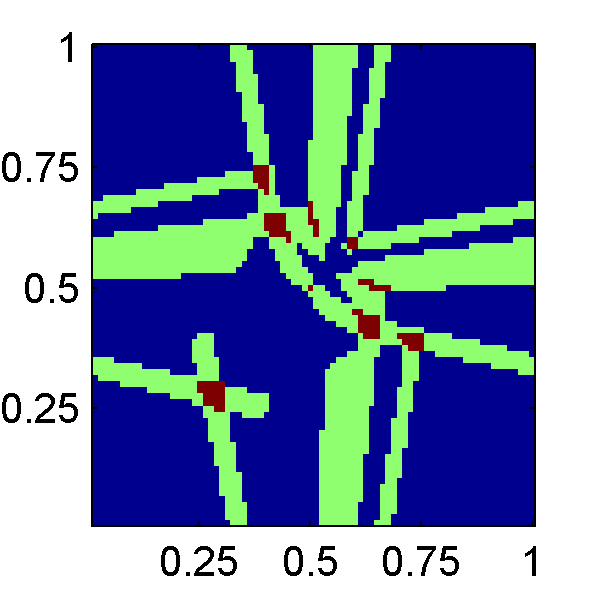}
\includegraphics[width=5cm]{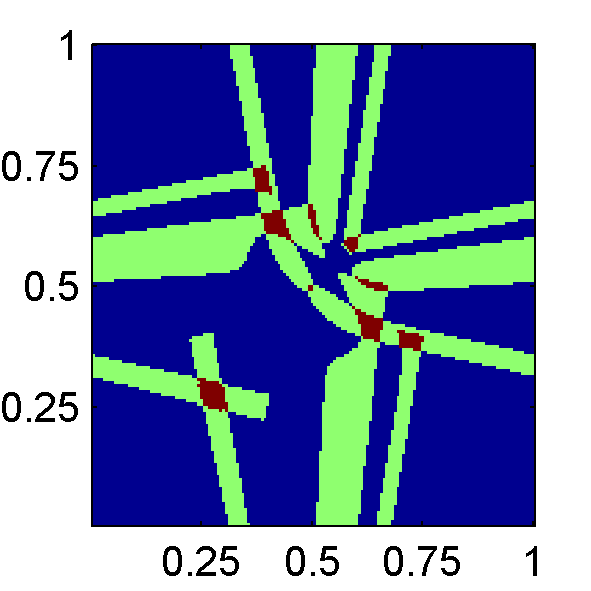}
\caption{
In Example 3, the number of phases obtained in the $\textsc{PhaseSep}$ procedure for each tile $K$, i.e. $\abs{\widetilde{\Theta}_{K}}$.
These figures are corresponding to $1/h = 10, 20, 40, 80$ and $120$, respectively.
Each tile in these figures represent the number of phase (Blue: 2, Green: 3 Red: 4).
The horizontal and vertical axis represent the $x_1$ and $x_2$ coordinates.}
\label{eg3ndof}
\end{figure}

\begin{figure}[p]
\centering
\includegraphics[height=5cm]{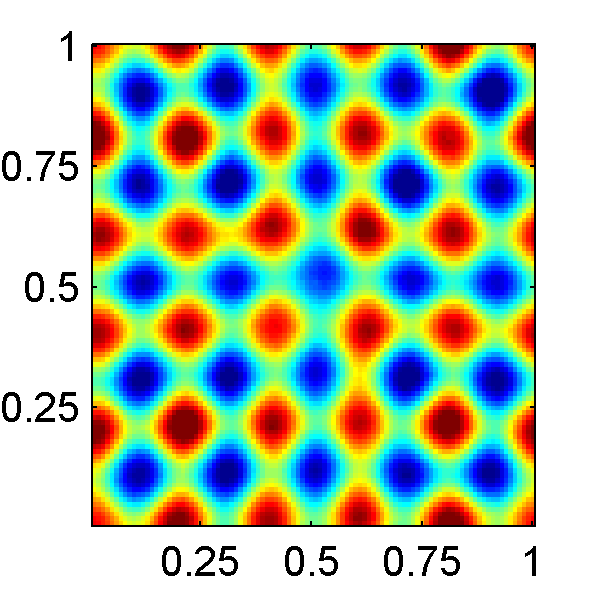}
\includegraphics[height=5cm]{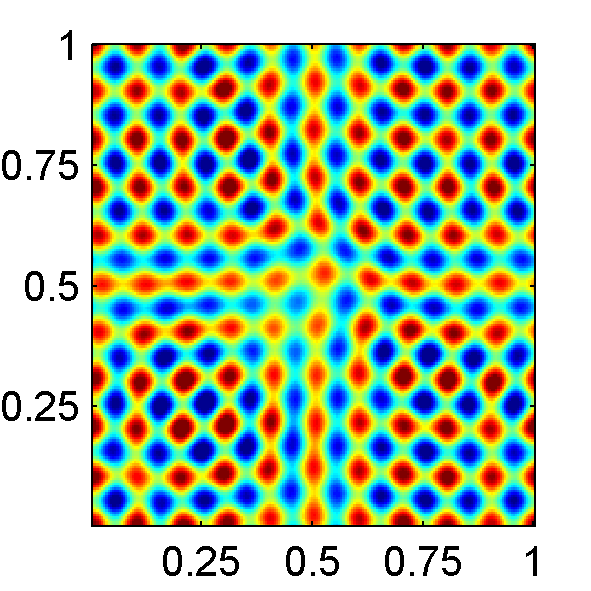}
\includegraphics[height=5cm]{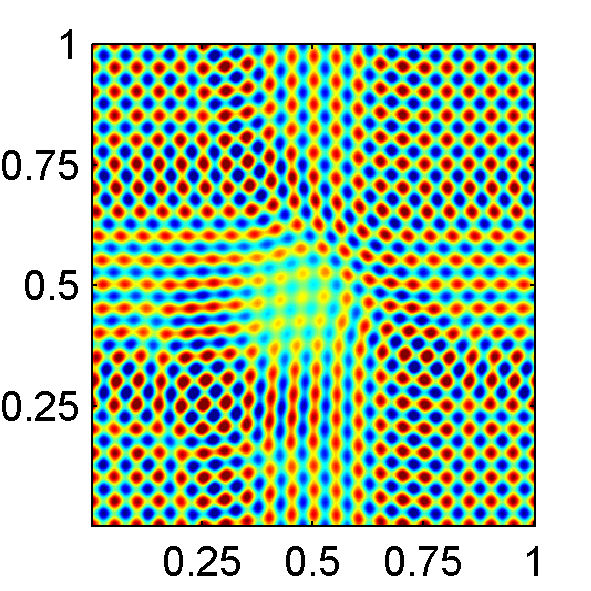}
\includegraphics[height=5cm]{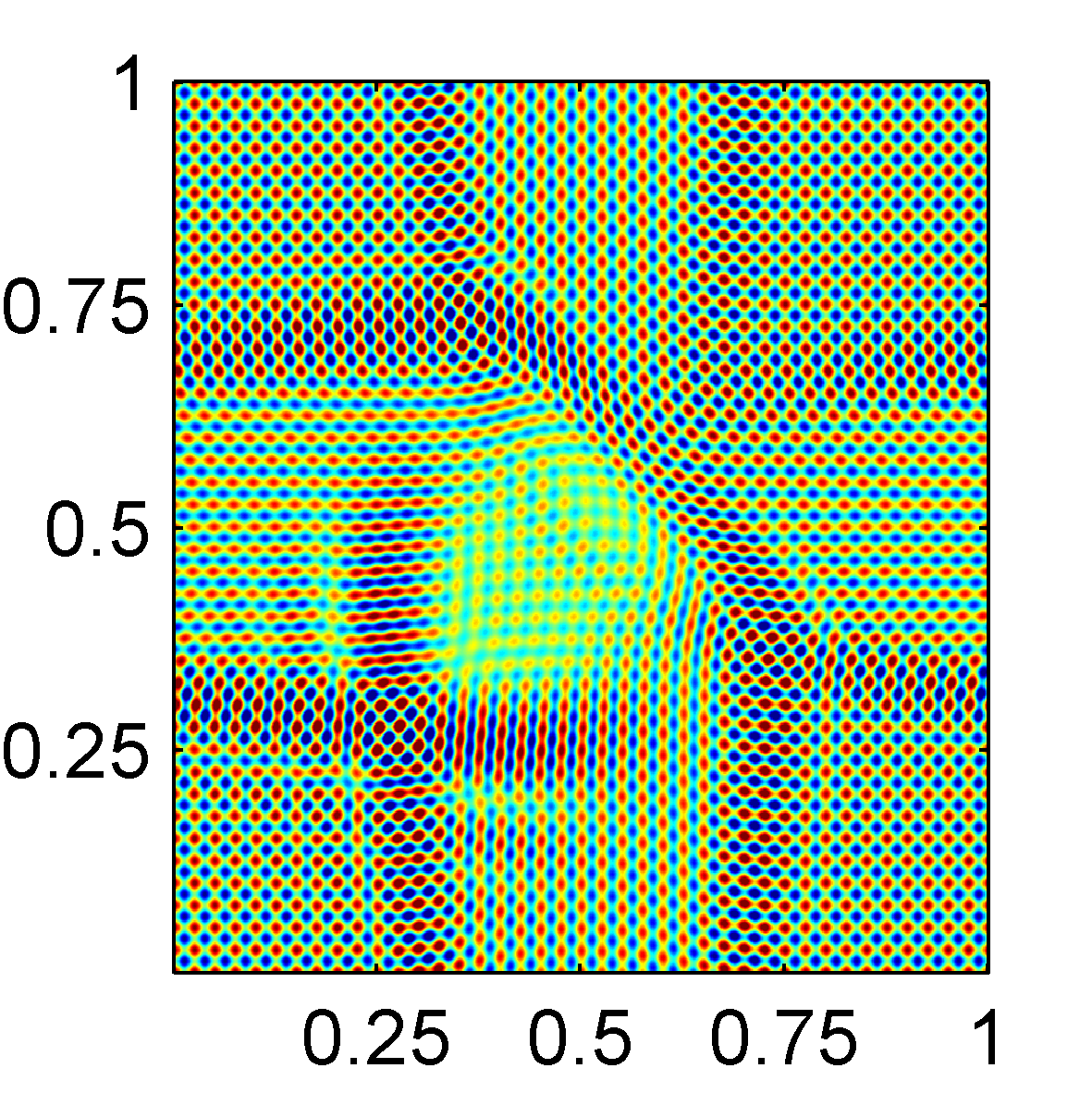}
\includegraphics[height=5cm]{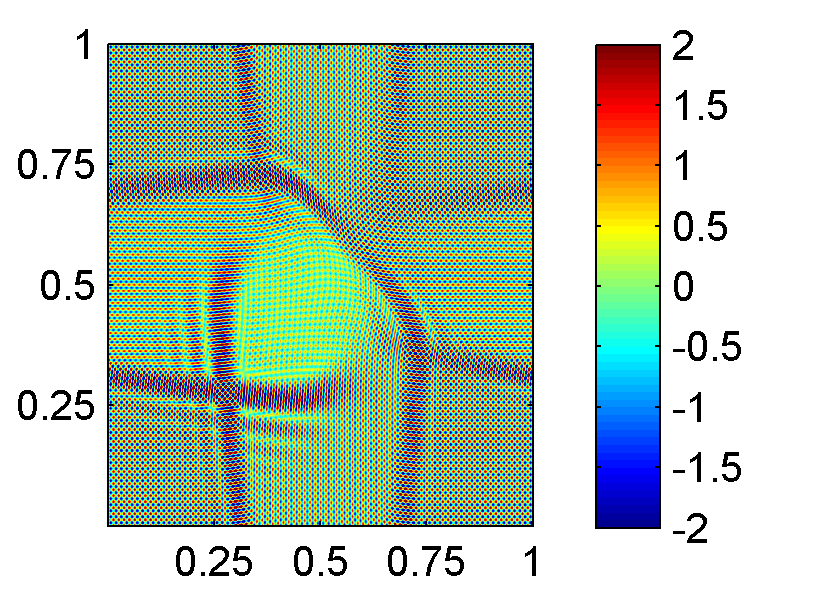}
\caption{In Example 3, the real part of the reference solutions for $1/h = 10, 20, 40, 80$ and $120$, respectively. The horizontal and vertical axis represent the $x_1$ and $x_2$ coordinates.}
\label{eg3ps}
\end{figure}
\begin{figure}[p]
\centering
\includegraphics[height=5cm]{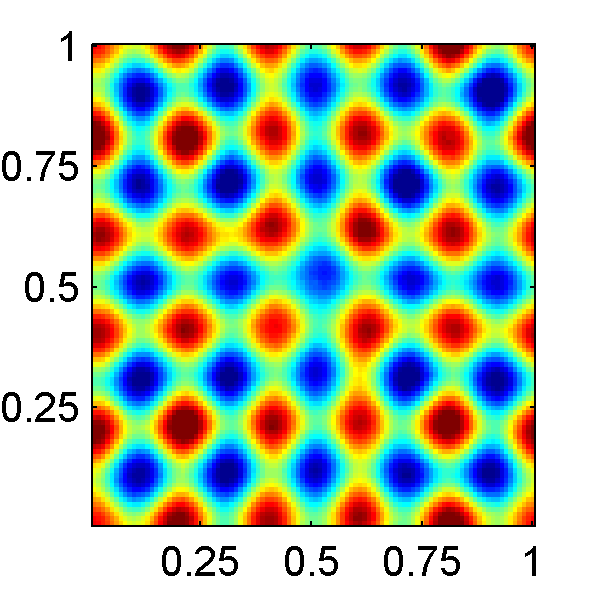}
\includegraphics[height=5cm]{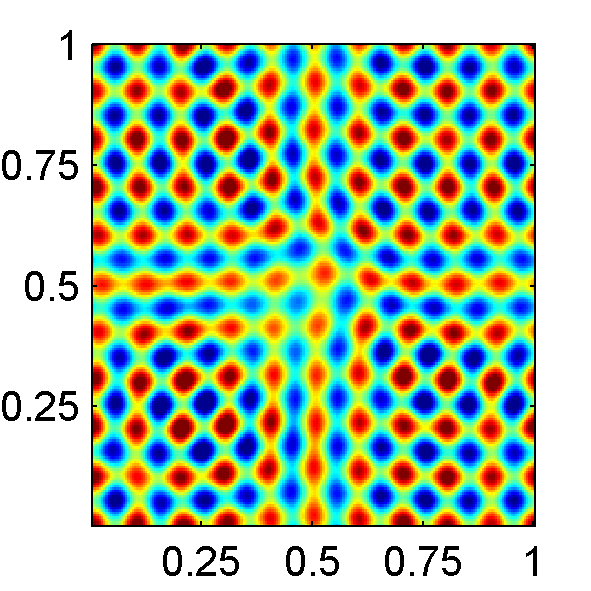}
\includegraphics[height=5cm]{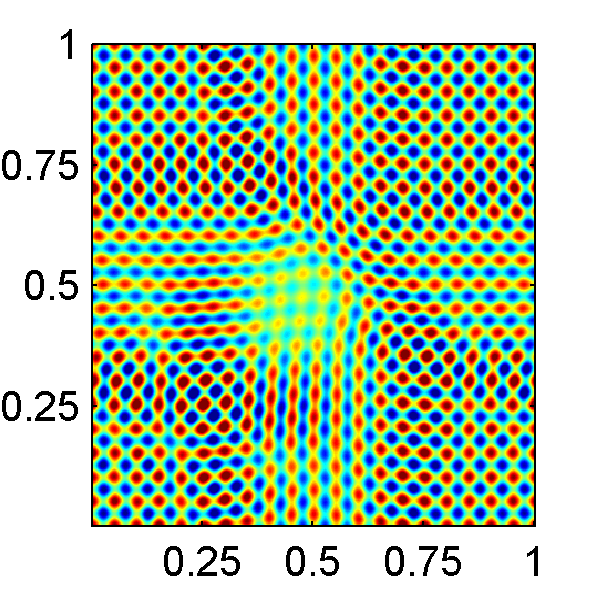}
\includegraphics[height=5cm]{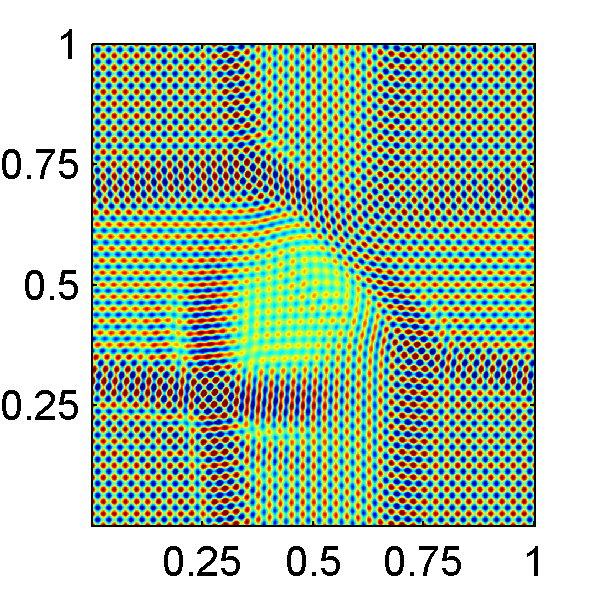}
\includegraphics[height=5cm]{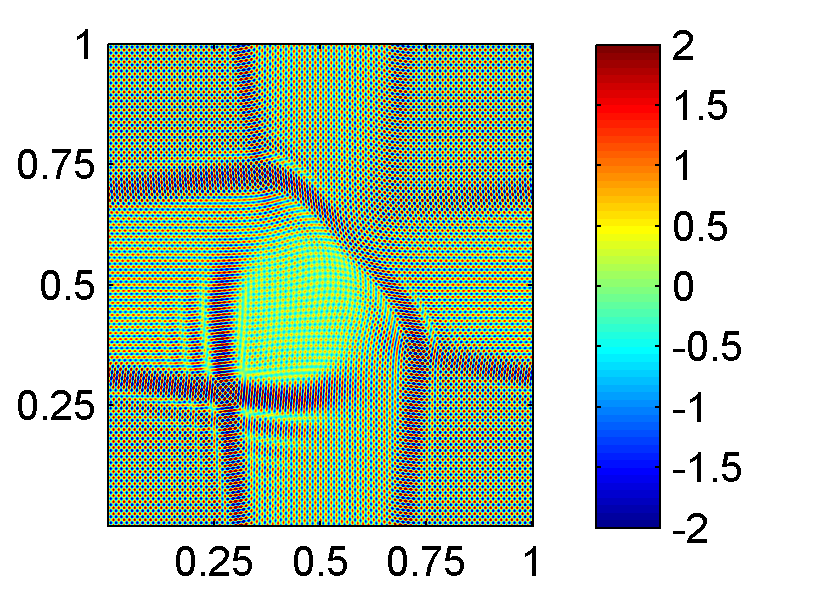}
\caption{In Example 3, the real part of the POD solutions for $1/h = 10, 20, 40, 80$ and $120$, respectively. The horizontal and vertical axis represent the $x_1$ and $x_2$ coordinates.}
\label{eg3dg}
\end{figure}
\begin{figure}[p]
\centering
\includegraphics[height=5cm]{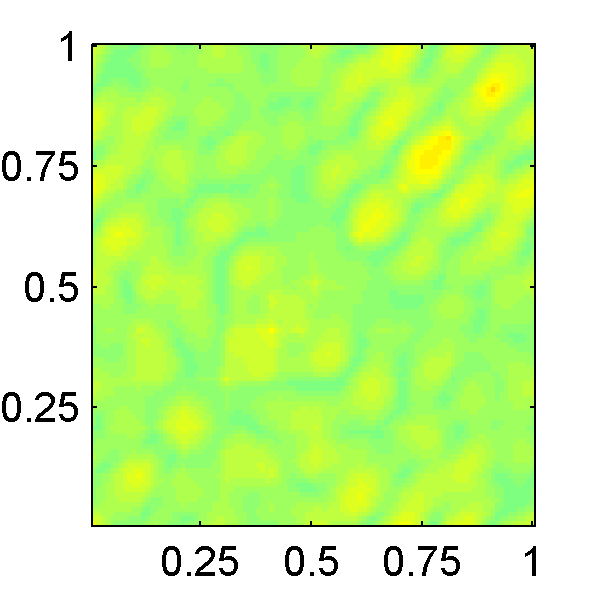}
\includegraphics[height=5cm]{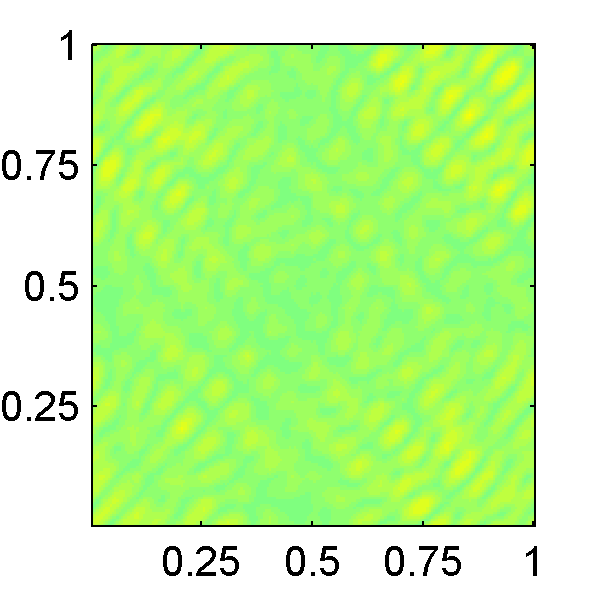}
\includegraphics[height=5cm]{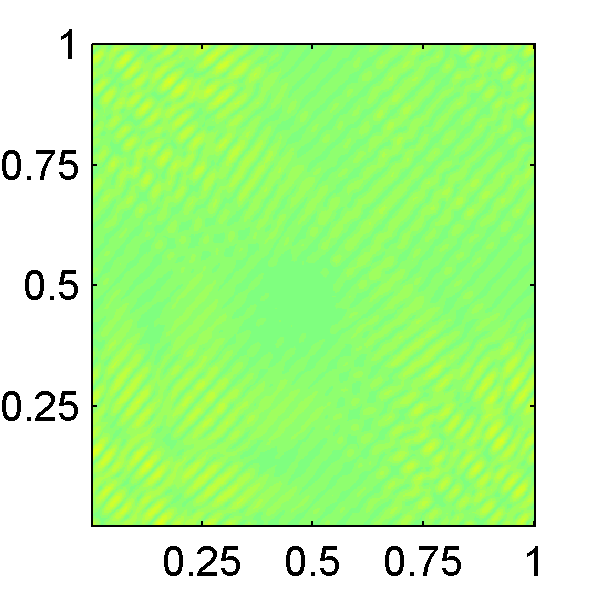}
\includegraphics[height=5cm]{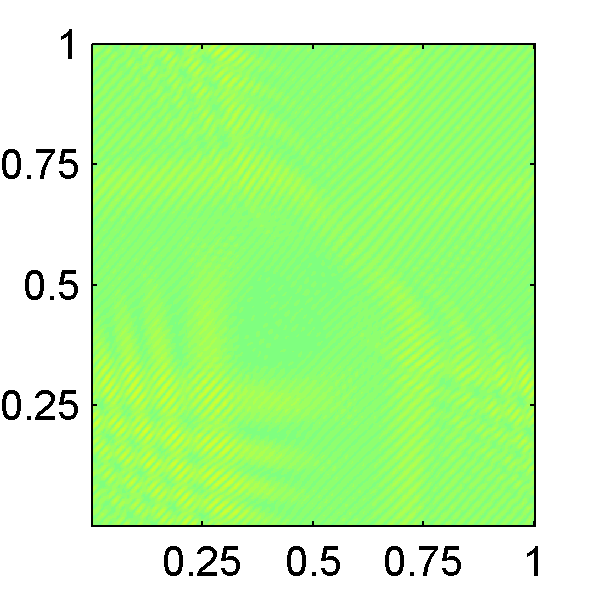}
\includegraphics[height=5cm]{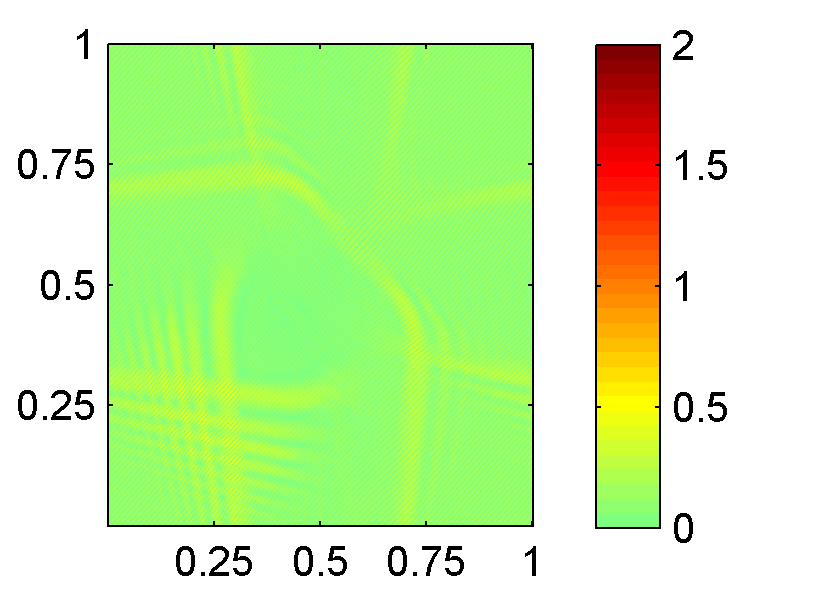}
\caption{In Example 3, the absolute difference between the
POD solutions and the reference solutions in Example 3 for $1/h = 10, 20, 40, 80$ and $120$, respectively. The horizontal and vertical axis represent the $x_1$ and $x_2$ coordinates.}
\label{eg3diff}
\end{figure}

\begin{figure}[p]
\centering
\includegraphics[width=5cm]{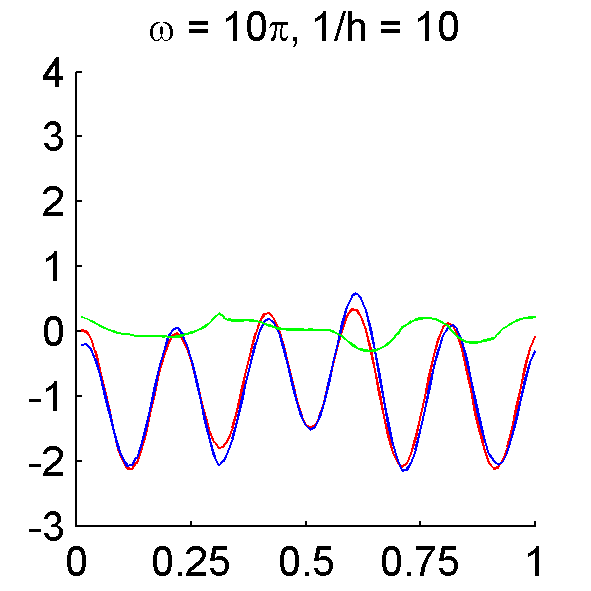}
\includegraphics[width=5cm]{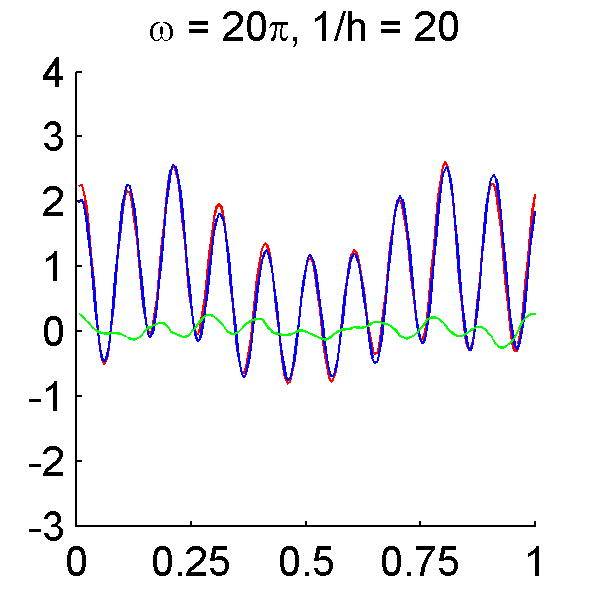}
\includegraphics[width=5cm]{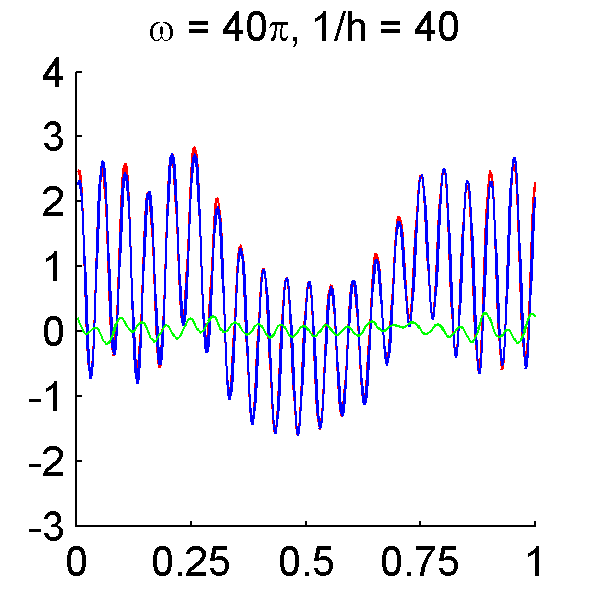}
\includegraphics[width=5cm]{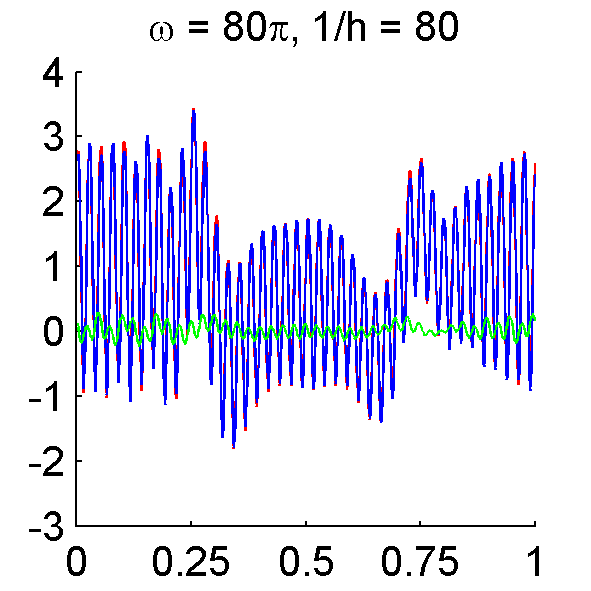}
\includegraphics[width=10cm]{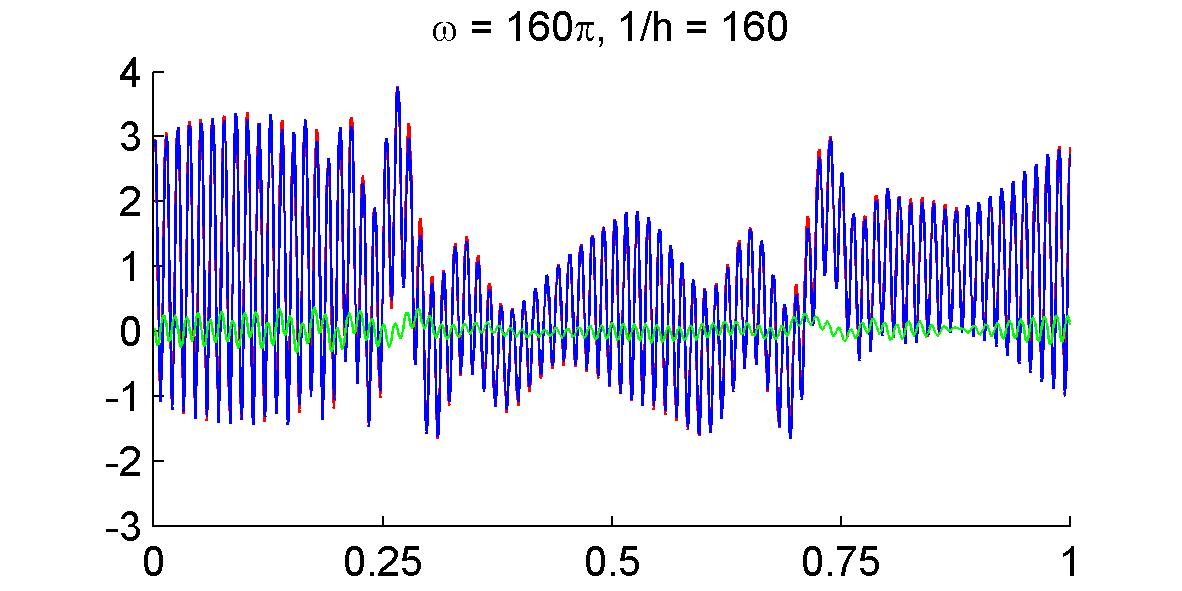}
\caption{ In Exmaple 3, the real part of solutions obtained from our method (blue),
the reference solutions (red) and the difference between these solutions (green)
at $T=1$ and $x_2 = 0.3$
under different settings of mesh size $h$ and the parameter $\omega$.
The horizontal axis represent the $x_1$ coordinate.}
\label{eg3waveform}
\end{figure}
\subsection{Example 4}
In this example, we consider a periodic velocity and observe the behavior of our approximate solution as $\omega\to\infty$ while keeping $\omega h = \bigO{1}$, using the POD system.
We consider the wave speed $c_2$, and take the initial solution so that $L=1$,
$\phi_1(x) = x_1$,
$A_1(x) = 1$ and $B_1(x,\omega) = -i\omega$. 
All other parameters are taken as before.


In Table \ref{eg4tblerr} we list the relative errors of the approximate solutions.
We also compare the condition numbers and degrees of freedom of the POD systems to those of the original systems.
We observe that the POD systems has a better conditioning with number of degrees of freedom about the same with the original systems.
We see that the relative $L^2$-error of the numerical solution for the POD system does not increase significantly as the $\omega$ increase.

\begin{table}[h]
\centering
\begin{tabular}{|c|c|c|c|c|c|c|c|}
\hline
\multirow{2}{*}{$\omega$} &
\multirow{2}{*}{$1/h$} &
\multirow{2}{*}{$\omega h$} & 
\multicolumn{2}{|c|}{Condition number} &
\multicolumn{2}{|c|}{Number of dof} & Relative \\ \cline{4-7}
 &  &  & 
original & POD & original & POD
& $L^2$ error (\%)\\ \hline
$10 \pi$ & 10 & \multirow{5}{*}{$\pi$} & 1.37e+10 & 8.04e+07 & 1256 & 940 & 8.58\\ 
$20 \pi$ & 20 & & 6.76e+11 & 4.99e+07 & 4880 & 3654 & 8.82 \\ 
$40 \pi$ & 40 & & 1.59e+11 & 1.46e+08 & 19792 & 15200 & 7.67 \\ 
$80 \pi$ & 80 & & 5.80e+12 & 1.85e+08 & 79244 & 59535 & 9.37 \\ 
$160 \pi$ & 160 & & 8.77e+12 & 1.82e+08 & 319592 & 240167 & 10.33 \\ \hline
\end{tabular}
\caption{ In Example 4, the condition number
and the number of degrees of freedom of the original system and the POD systems,
and the relative $L^2$ error obtained by comparing the solution of the POD systems to the reference solution, under different settings of $\omega$ and meshsize $h$.}
\label{eg4tblerr}
\end{table}
We also demonstrate the quality of the solution by comparing our solution to the reference solutions in Figure~\ref{eg4ps} and \ref{eg4dg}.
In Figure~\ref{eg4diff}, we compute the absolute difference between our solutions and the reference solutions.
In Figure~\ref{eg4waveform}, we compare the real parts of our solutions to the reference solutions at $T = 1$ and $x_2 = 0.3$.

\begin{figure}[p]
\centering
\includegraphics[height=5cm]{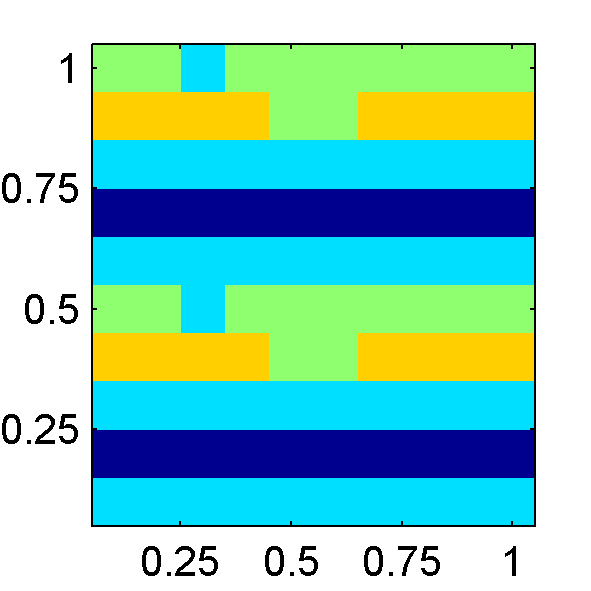}
\includegraphics[height=5cm]{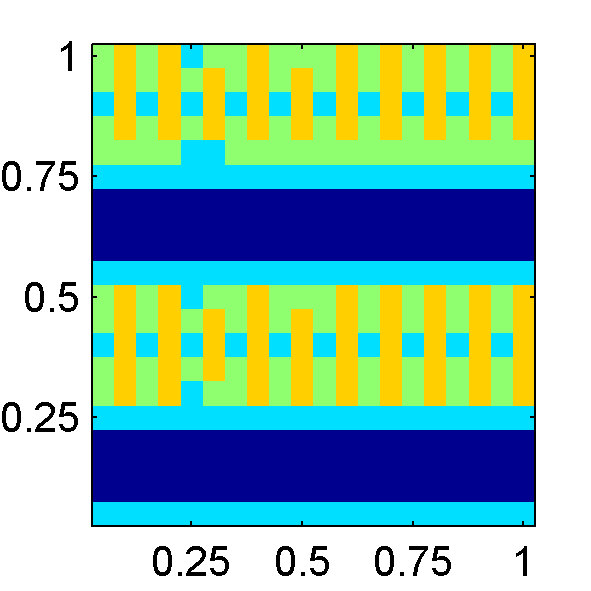}
\includegraphics[height=5cm]{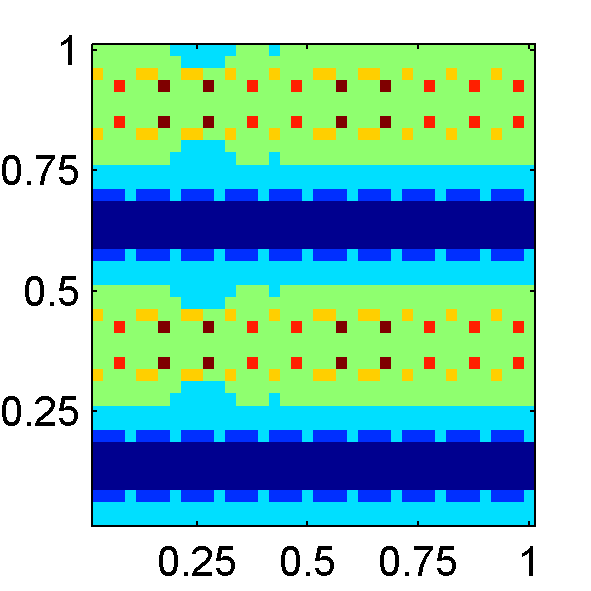}
\includegraphics[height=5cm]{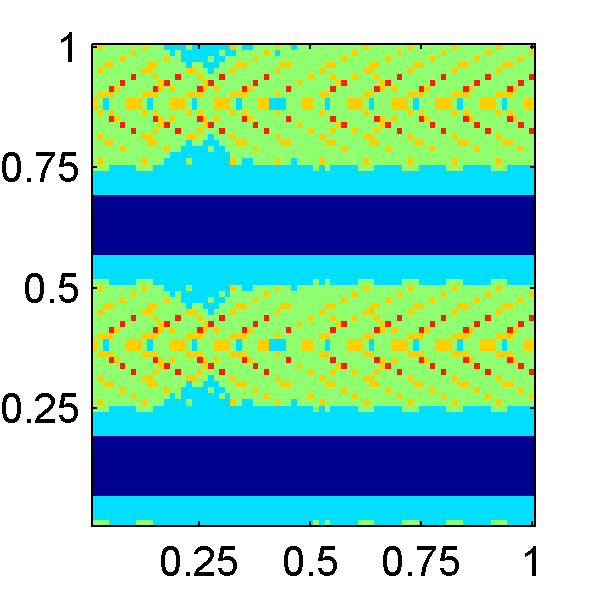}
\includegraphics[height=5cm]{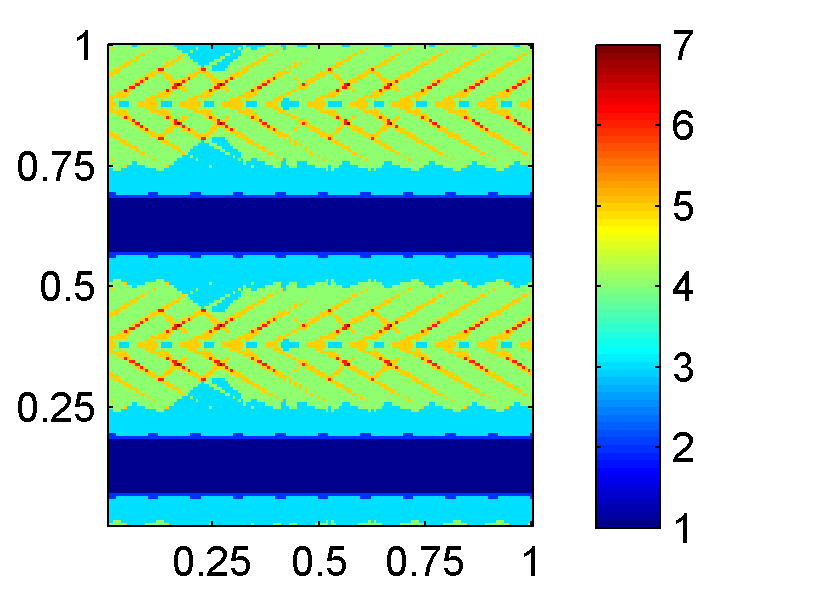}
\caption{
In Example 4, the number of phases obtained in the $\textsc{PhaseSep}$ procedure for each tile $K$, i.e. $\abs{\widetilde{\Theta}_{K}}$.
These figures are corresponding to $1/h = 10, 20, 40, 80$ and $120$, respectively.
Each tile in these figures represent the number of phase.
The horizontal and vertical axis represent the $x_1$ and $x_2$ coordinates.}
\label{eg4ndof}
\end{figure}

\begin{figure}[p]
\centering
\includegraphics[height=5cm]{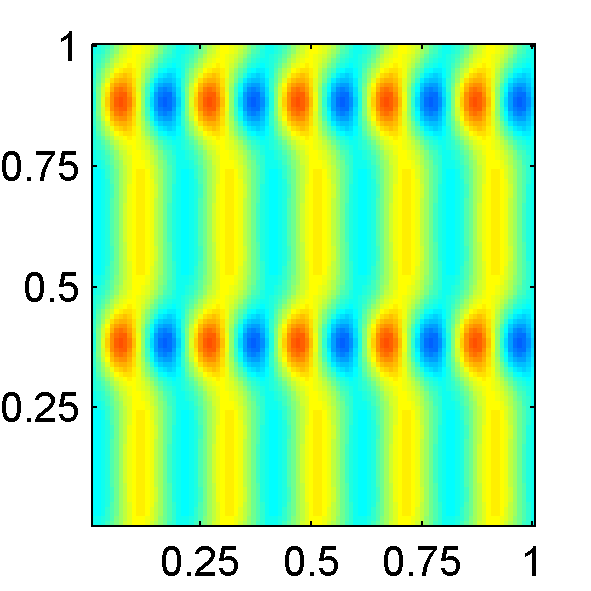}
\includegraphics[height=5cm]{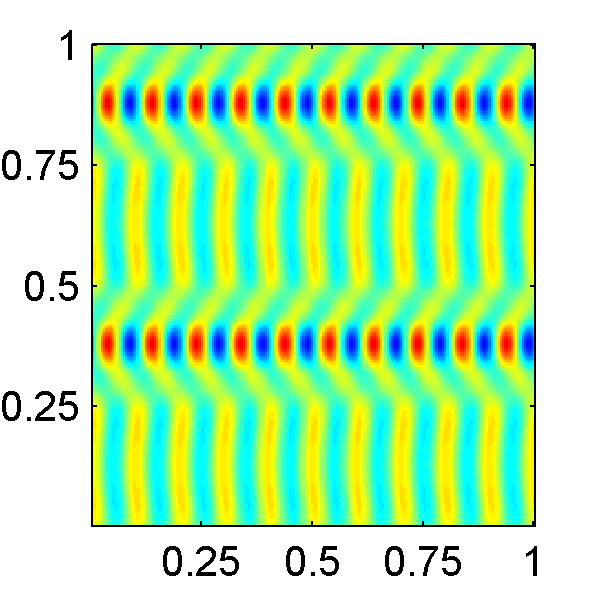}
\includegraphics[height=5cm]{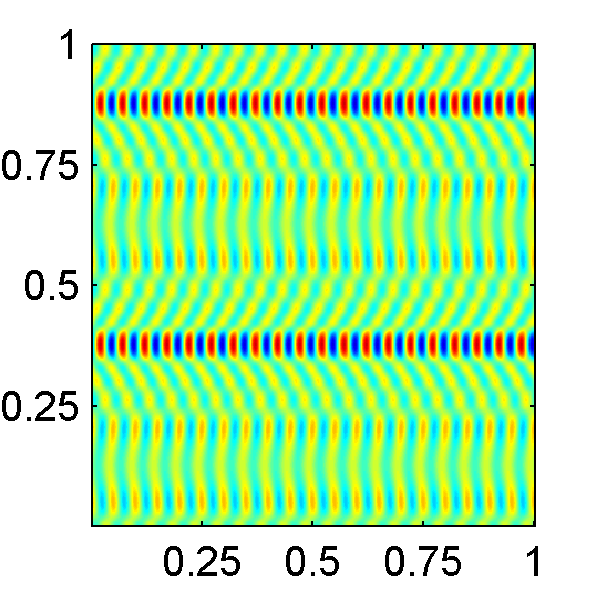}
\includegraphics[height=5cm]{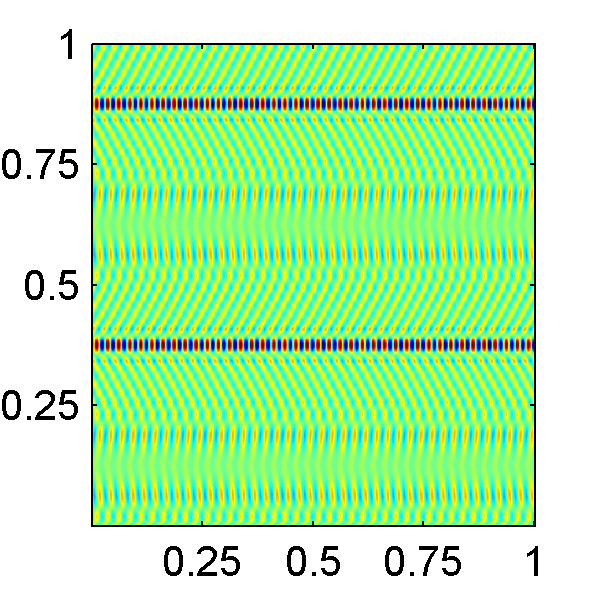}
\includegraphics[height=5cm]{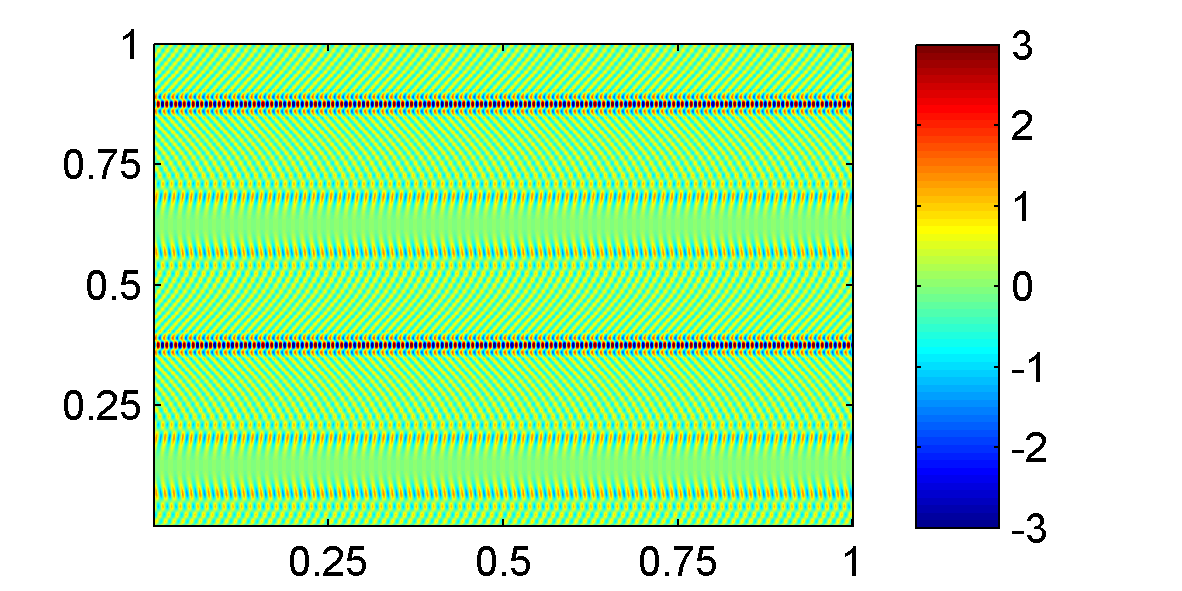}
\caption{In Example 4, the
real part of the reference solutions for $1/h = 10, 20, 40, 80$ and $120$, respectively. The horizontal and vertical axis represent the $x_1$ and $x_2$ coordinates.}
\label{eg4ps}
\end{figure}

\begin{figure}[p]
\centering
\includegraphics[height=5cm]{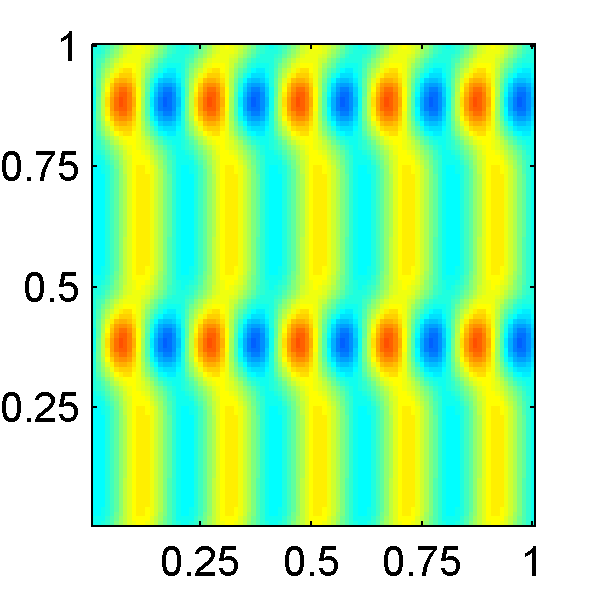}
\includegraphics[height=5cm]{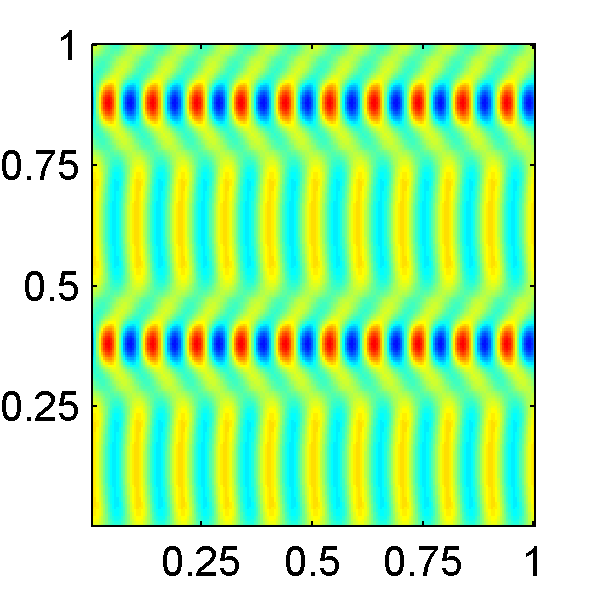}
\includegraphics[height=5cm]{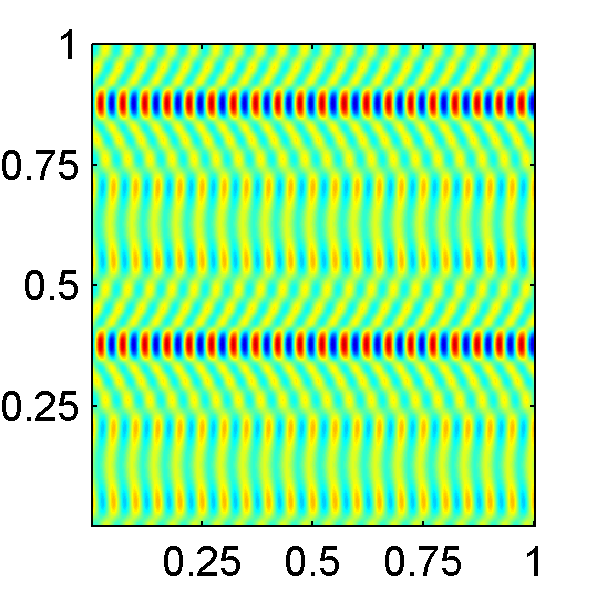}
\includegraphics[height=5cm]{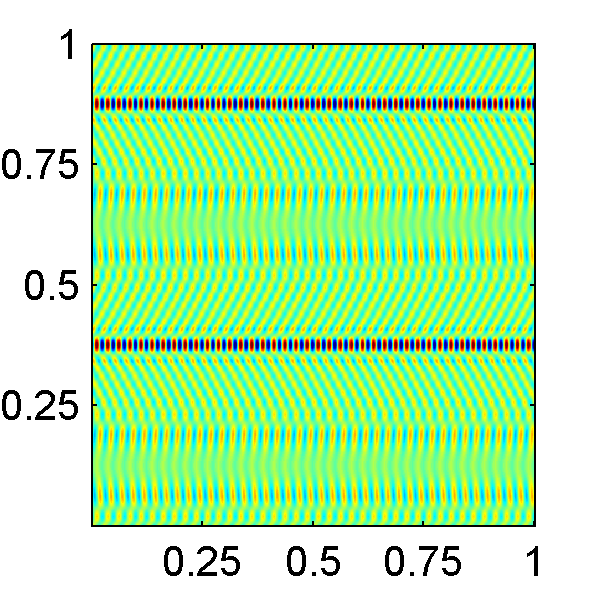}
\includegraphics[height=5cm]{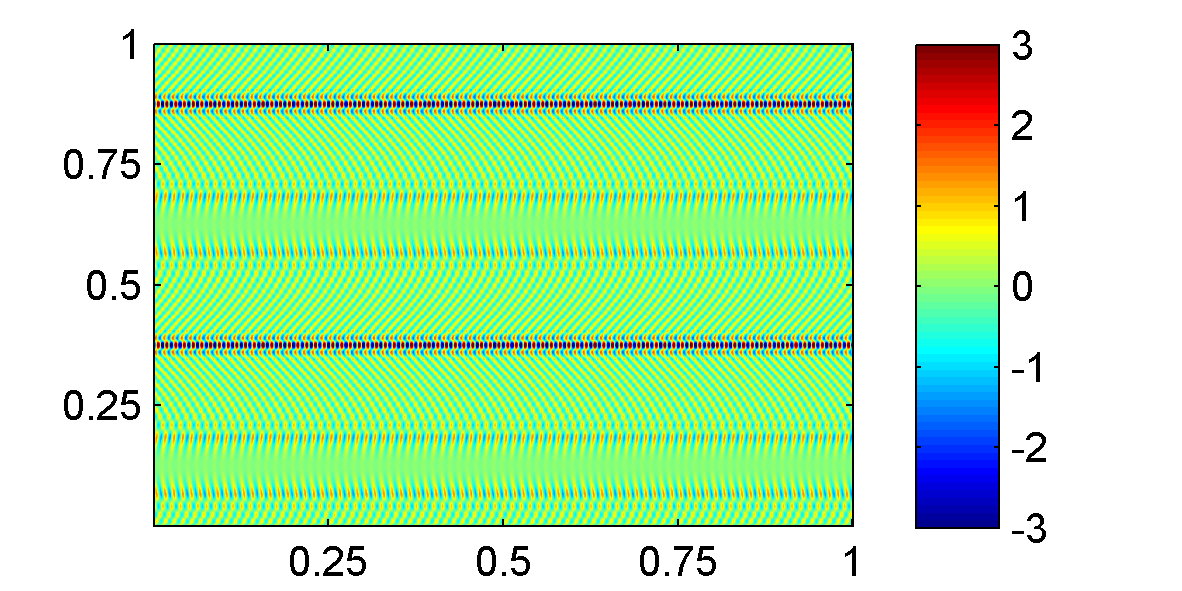}
\caption{In Example 4,
real part of the POD solutions for $1/h = 10, 20, 40, 80$ and $120$, respectively. The horizontal and vertical axis represent the $x_1$ and $x_2$ coordinates.}
\label{eg4dg}
\end{figure}

\begin{figure}[p]
\centering
\includegraphics[height=5cm]{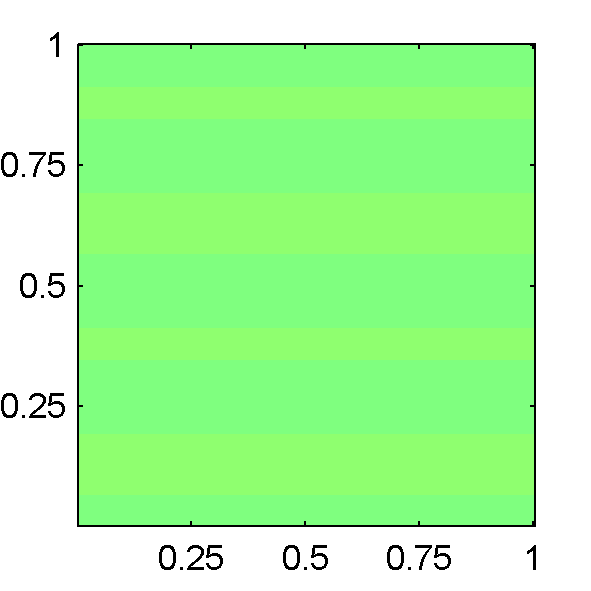}
\includegraphics[height=5cm]{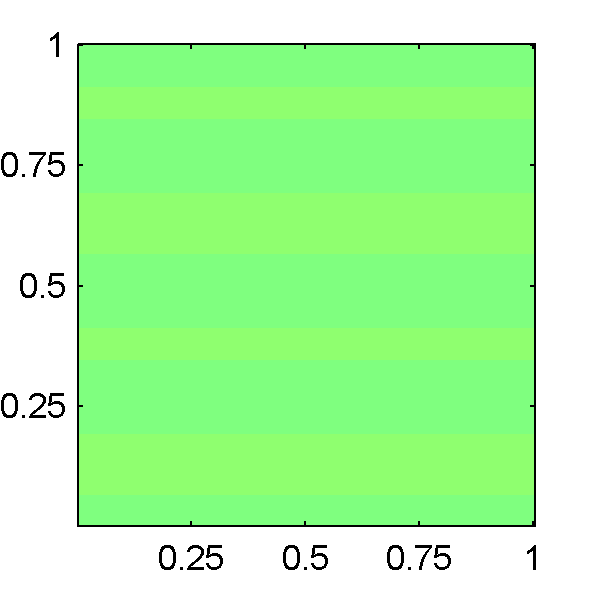}
\includegraphics[height=5cm]{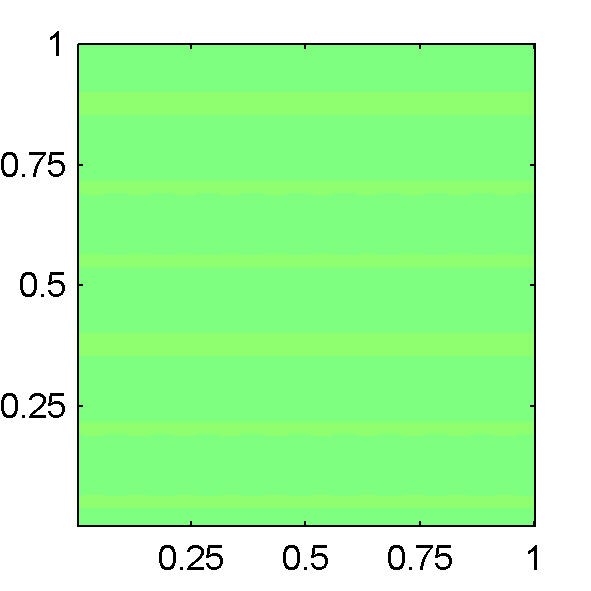}
\includegraphics[height=5cm]{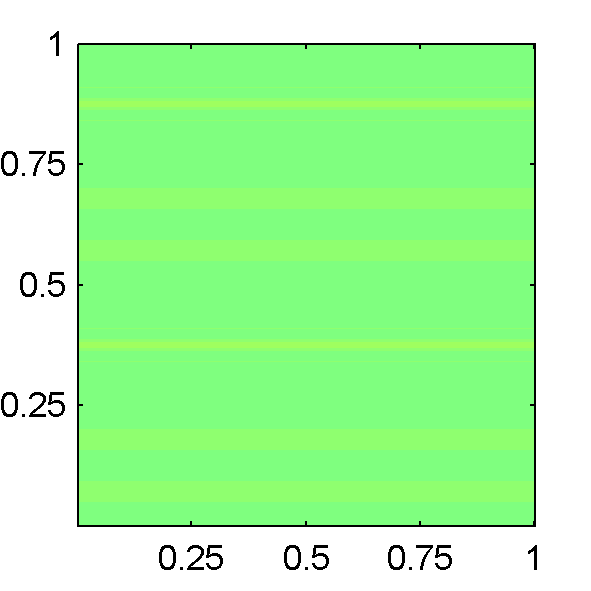}
\includegraphics[height=5cm]{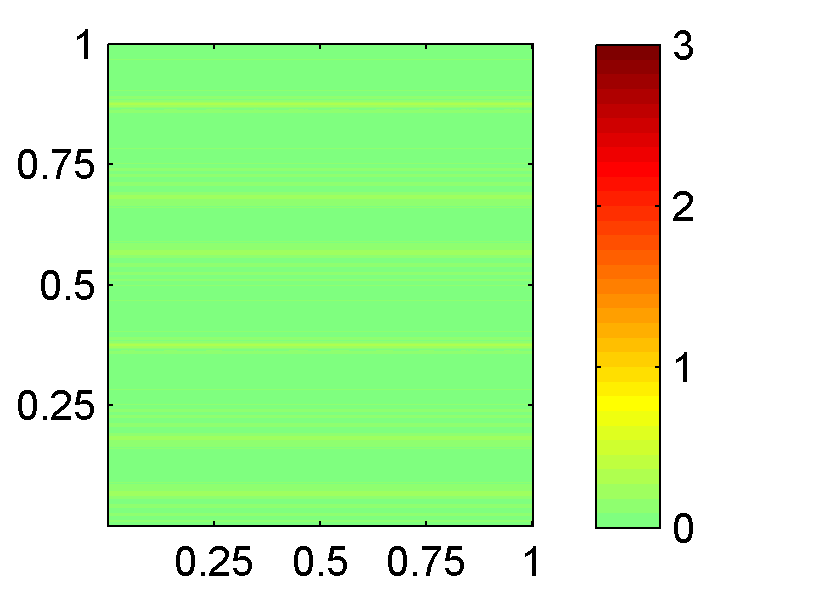}
\caption{In Example 4, 
This figure shows the absolute difference between the
POD solutions and the reference solutions for $1/h = 10, 20, 40, 80$ and $120$, respectively. The horizontal and vertical axis represent the $x_1$ and $x_2$ coordinates.}
\label{eg4diff}
\end{figure}

\begin{figure}[p]
\centering
\includegraphics[height=5cm]{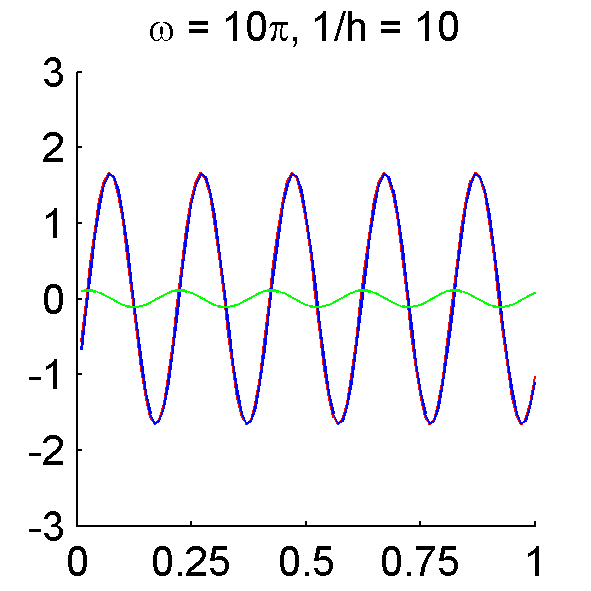}
\includegraphics[height=5cm]{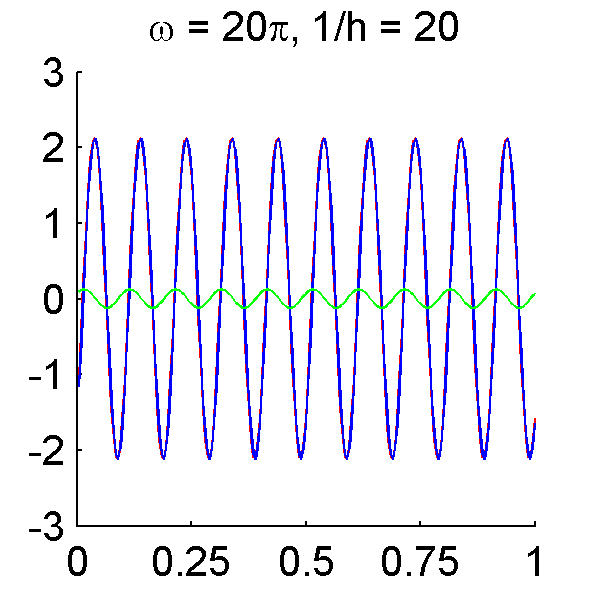}
\includegraphics[height=5cm]{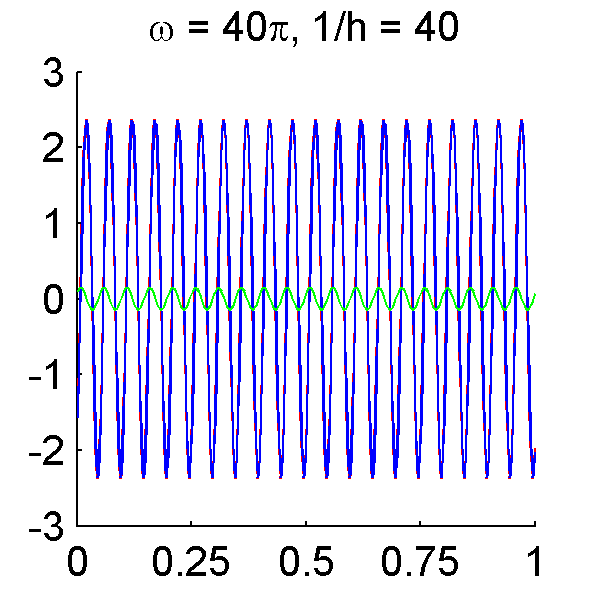}
\includegraphics[height=5cm]{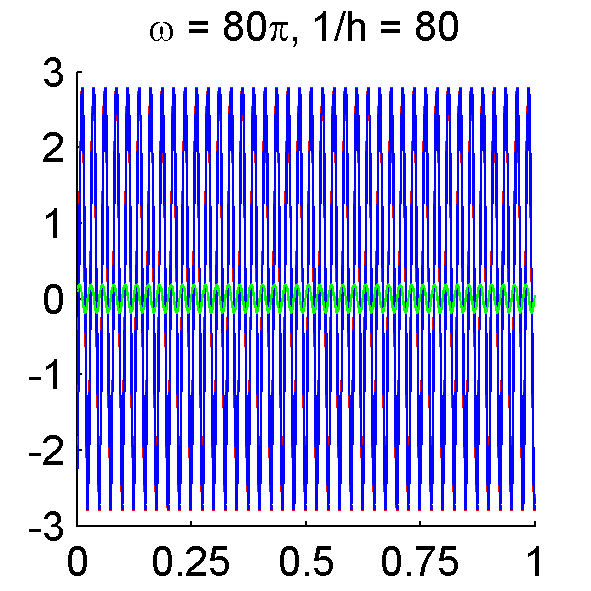}
\includegraphics[height=5cm]{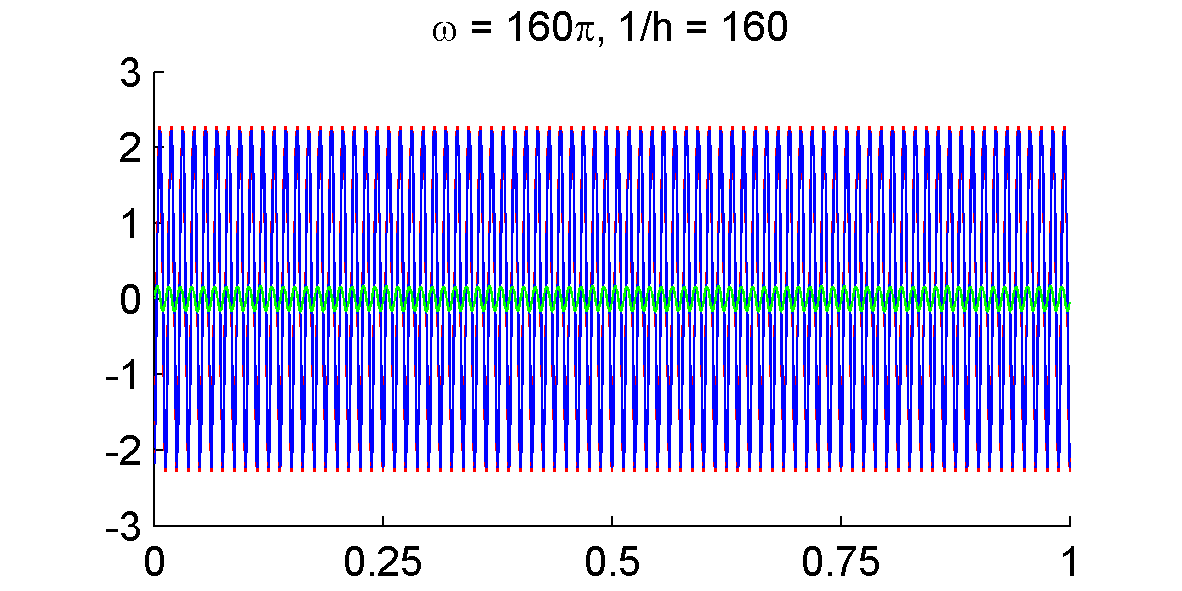}
\caption{ In Example 4, the real part of solutions obtained from our method (blue),
the reference solutions (red) and the difference between these solutions (green)
at $T=1$ and $x_2 = 0.37$
under different settings of mesh size $h$ and the parameter $\omega$.
The horizontal axis represent the $x_1$ coordinate.}
\label{eg4waveform}
\end{figure}

\section{Conclusion}
We proposed a new method for solving the time-domain acoustic wave propagation problem with smoothly varying inhomogeneous media in the high-frequency regime. 
Our method is based on plane-wave type basis functions, with the phases carefully computed by using ideas
from geometrical optics, wavefront tracking, and dimensional reduction. 
The numerical results show evidence that the accuracy of the solution of our proposed method is dramatically better than standard IPDG method and the proposed method's accuracy do not worsen significantly as frequency increases, in contrast to standard IPDG method with polynomial basis.
In the future, we plan to design numerical schemes for more complicated heterogeneous and multiscale media
by combining the approach in this paper and some recent ideas on multiscale schemes \cite{chung2016adaptive}. 

\section*{Acknowledgement}
Chung is partially supported by Hong Kong RGC General Research Fund (Project: 14301314)
and CUHK Direct Grant for Research 2016-17.
 Qian is partially supported by NSF grants (1522249 and 1614566).
 
\bibliography{ref,myref}{}
\bibliographystyle{abbrv}
\end{document}